\RequirePackage{fix-cm}
\documentclass{amsart}

\usepackage{mathabx}

\usepackage{graphicx}
\usepackage{mathrsfs}
\usepackage{amsfonts,amsmath}

\usepackage{color}
\usepackage{hyperref}

\renewcommand{\H}[1]{{\textbf{[#1]}}}
\renewcommand{\P}{{\mathbb P}}

\renewcommand{\S}{{\mathcal S}}

\newcommand{\B}{{\mathbb B}}

\newcommand{\C}{{\mathbb C}}
\newcommand{\R}{{\mathbb R}}

\newcommand{\N}{{\mathbb N}}

\newcommand{\I}{{\mathbb I}}

\newcommand{\atanh}{{\rm atanh}}
\newcommand{\supp}{{\rm supp}}

\renewcommand{\c}{{\mathbf c}}

\newcommand{\E}{{\mathbb E}}

\newcommand{\F}{{\mathcal F}}

\newcommand{\leftB}{{{[\![}}}
\newcommand{\rightB}{{{]\!]}}}
\renewcommand{\S}{{\mathbb S}}

\theoremstyle{plain}
\newtheorem{theorem}{Theorem}[section]

\newtheorem{proposition}[theorem]{Proposition}

\theoremstyle{definition}
\newtheorem{remark}{Remark}[section]

\newcommand{\mysection}{\setcounter{equation}{0} \section}

\begin{document}

\title{Limit Theorems for random walks in the hyperbolic space}
\author{V. Konakov}

\address{Higher School of Economics, National Research University, Pokrovski Boulevard 11, Moscow, Russian Federation. 
}

\email{VKonakov@hse.ru}

\author{ S. Menozzi}
\address{ Laboratoire de Mod\'elisation Math\'ematique d'Evry\\ 
(LaMME), UMR CNRS 8071, Universit\'e d'Evry Val-d'Essonne, 23 Boulevard de France, 91037 Evry.
}
\email{stephane.menozzi@univ-evry.fr}

\keywords{Hyperbolic space, Random Walks and Brownian motion, Local  Limit Theorems.\\
\indent \textbf{MSC2020:}  60F05, 43A85}

\date{\today}

\thanks{For the first author, the article was prepare within the framework of the Basic Research Program at HSE University.}

\maketitle

\begin{abstract}
We  prove central and local limit theorems for random walks on the Poincaré hyperbolic space of dimension $n\ge 2 $.
To this end we use the ball model and describe the walk therein through the M\"obius addition and multiplication. This also allows to derive a corresponding law of large numbers.

\end{abstract}

\section{Introduction}
Let $\mathbb H^n $ ($n\ge 2) $ denote the real hyperbolic space of dimension $n\ge 2$. This is the complete and simply connected Riemannian manifold with constant negative sectional curvature equal to -1.

We will be interested in establishing limit theorems (central, local, law of large numbers) for a corresponding underlying random walk. 
In the case $n=2,3$, the central limit theorem was first obtained for the models of the Poincaré-Lobachevski ball and plane in the seminal paper by Karpelevich \textit{et al.} \cite{karp:tutu:shur:59}. It was shown therein that the suitably renormalized sum converges in law to the \textit{normal} law, corresponding to the fundamental solution of the heat equation involving the Laplace-Beltrami operator, at a certain time that could be seen as a \textit{variance} in this setting.
We can as well refer to the monograph of Terras \cite{terr:13}, Section 3.2 for additional details. 

In \cite{karp:tutu:shur:59} the M\"obius addition already appeared in order to provide a non-Euclidean analogue of the \textit{sum} of random variables.
In a more general setting, the connection with Harmonic analysis on the underlying appropriate structure, like e.g. Lie groups, provided a natural way to define a corresponding random walk and to investigate the associated limit theorems. Indeed, the harmonic analysis somehow provides natural extensions of the characteristic functions, variances. We again refer to \cite{terr:13} and \cite{terr:88} for results in that direction.

The above results concentrate on the behavior of a renormalized sum.
Concerning a more functional approach we can mention the works by Pinsky (see \cite{pins:75} and \cite{pins:76}) which established how the so-called \textit{isotropic transport process}\footnote{This process is somehow very natural to consider. Its construction starting from a point $x$ in a complete Riemannian manifold $M$ consists in choosing uniformly a direction on the unit sphere of the tangent space $T_{x}(M) $. One then follows the geodesic line starting from $x_0$ for the chosen direction for an appropriate time and then reiterate.} actually converges in law, when suitably normalized, to the Brownian motion on the manifold when the Ricci curvature is bounded from below. The case of a Poincaré plane when the geodesic ball is also discretized, thus leading to a space-time discrete approximation, is discussed in Gruet \cite{grue:08}.

We will here focus on the central limit theorem, the local limit (we will as well establish as a by-product of our analysis a law of large numbers) for any dimension $n$ of the hyperbolic space. Whereas local limit theorems were previously established for some nilpotent groups, see e.g. Breuillard \cite{breu:05} for the Heisenberg group, this seems to be, to the best of our knowledge, new in the current setting. To this end we will rely on the Poincaré ball model and use related harmonic analysis tools.

The paper is organized as follows. We describe the random walk and the related convergence results in Section \ref{MAIN_RES}. We provide in Section \ref{SEC_FOUR} the main tools from harmonic analysis needed to perform the analysis. Eventually, Section \ref{PROOFS} is dedicated to the proof of the main results. 

\section{Main Results}
\label{MAIN_RES}
\subsection{Setting} 
Let $\mathbb H^n $ ($n\ge 2) $ denote the real hyperbolic space of dimension $n\ge 2$. This is the complete and simply connected Riemannian manifold with constant negative sectional curvature equal to -1.
We will consider throughout the document the Poincar\'e ball model for the  hyperbolic space $\mathbb H^n $ ($n\ge 2) $. Namely, we introduce:
\begin{equation*}
\mathbb B^n:=\{x\in \R^n: \|x\|< 1 \},
\end{equation*}
where $\|\cdot\| $ stands for the Euclidean norm, which we endow with the metric 
\begin{equation*}
ds^2=\frac{4\big(dx_1^2+\cdots dx_n^2 \big)}{(1-\|x\|^2)^2}
\end{equation*}
and the corresponding Riemannian volume
\begin{equation}
\label{VOLUME}
\mu_{\mathbb B^n}(dx)=2^n(1-\|x\|^2)^{-n} dx.
\end{equation}
We recall that the hyperbolic  distance to the origin writes for $z\in \mathbb B^n $:
\begin{equation*}
\eta=2\ \atanh(\|z\|)=\log \Big(\frac{1+\|z\|}{1-\|z\|} \Big).\label{RADIAL_HYPER},
\end{equation*} 
where we use the notation $\eta$ for coherence with \cite{karp:tutu:shur:59}.

Hence,  for $z\in \B^n$,
$$z=r \Theta =\tanh(\frac \eta 2) \Theta, \ \Theta\in \mathbb S^{d-1}.$$ 
The coordinates $(r,\Theta)\in [0,1)\times \mathbb S^{d-1}$ are the \textit{usual} polar coordinates and $(\eta,\Theta)\in [0,+\infty)\times \mathbb S^{d-1} $ the geodesic ones. \textcolor{black}{In the geodesic coordinates, we have:
\begin{equation}\label{RADIAL_GE0_MEAS}
\mu_{\mathbb B^n}(dx) = \sinh(\eta)^{n-1}d\eta  \Lambda(d\Theta),
\end{equation}
where $ \Lambda$ stands for the Lebesgue measure of the sphere $\mathbb S^{n-1} $}.

 To define a random walk on $\B^n $ we will introduce the Möbius addition  $(x,y)\in \B^n\times \B^n\mapsto \B^n  $ which writes in Euclidean coordinates:
 \begin{equation}\label{Mobius_ADD}
x\oplus y:=\frac{(1+2\langle x,y\rangle+\|y\|^2) x+(1-\|x\|^2)y}{1+2\langle x,y\rangle+\|x\|^2\|y\|^2},
 \end{equation} 
where $\langle\cdot,\cdot\rangle $ stand for the \textit{usual} Euclidean scalar product on $\R^n $.

\textcolor{black}{We will also use the left inverse, i.e. for all $x\in \B^n$, \eqref{Mobius_ADD} directly gives that  $\ominus x:=-x $ satisfies $\ominus x\oplus x=0 $}.

We refer to \cite{ahlf:81} for a thorough presentation of related operations and properties (see also Section \ref{SEC_FOUR} below). Let us also mention that the operation \eqref{Mobius_ADD} can also be written through the formalism of Clifford algebras, which then allows to extend the expression of the Möbius transform/addition in the complex plane to the current setting (see e.g. \cite{ferr:15}). 

We also introduce the Möbius multiplication, i.e. for all $\gamma\in \R $, $z\in \B^n $,
\begin{equation}\label{MULT_MOBIUS}
\gamma \otimes z:=\tanh(\gamma\ \atanh(\|z\|)) \frac{z}{\|z\|},
\end{equation}
and refer to \cite{bara:unga:20} for a thorough discussion about this operation and other operations on the unit ball connected with differential geometry.
\subsection{The random walk and associated convergence results}
Let $Z$ be a $\B^n $-valued random variable defined on some probability space $(\Omega,\mathcal A,\P) $. We will further assume that:
\begin{trivlist}
\item[]\H{\bf R} $Z$ has \textit{radial} density $f_Z\in C_0^\infty(\B^n) $ w.r.t. the Riemannian volume.
\end{trivlist}

Let now $(Z^j)_{j\ge 1} $ be a sequence of i.i.d random variables which have the same law as $Z $. Define then for $N\in \N$,
\begin{align}
\label{WALKS}
\bar S_N&:=\oplus_{j=1}^N \frac 1N \otimes Z^j,\notag\\
S_N&:=\oplus_{j=1}^N \frac 1{\sqrt N} \otimes Z^j.
\end{align}
Then the following results hold:
\begin{theorem}[Law of large numbers] \label{LLN} Under \H{R},
$$\bar S_N \overset{\P}{\underset{N}{\longrightarrow}} 0. $$
\end{theorem}

It will be shown in Section \ref{PROOFS} that $S_N$ has a density $f_{S_N}$, which can be expressed as the non Euclidean convolution of the densities of the $(Z^j)_{j\in \leftB 1, N\rightB}$. We quantify below the asymptotic behavior of that density.
\footnote{From now on we will denote by $\leftB\cdot, \cdot\rightB $ integer intervals.}
\begin{theorem}[Central limit theorem] \label{CLT} Under \H{R},
it holds that for measurable  sets $A$ in $\B^n $, 
 $$\int_{\B^n}\I_A(x) f_{S_{N}}(x) \mu_{\B^n} (dx)\underset{N}\longrightarrow \int_{\B^n}\I_A(x) \Psi(t,x) \mu_{\B^n}(dx),$$
 with 
\textcolor{black}{ $t:=\frac 1n  \int_0^{+\infty } \tilde \eta ^2  \mu_{Z,R}(d\tilde \eta)$, 
},
 where $\mu_{Z,R} $ stands for the  measure induced by the law of $Z$ in geodesic coordinates. 
 
 Also, $\Psi(t,\cdot)$  stands for the hyperbolic heat kernel in the model $\B^n $ for $\mathbb H^n$. Namely it denotes the fundamental solution of the equation
 $$ \frac 12 \Delta_{\B^n}\Psi (t,x)=\partial_t \Psi (t,x),\ \textcolor{black}{\Psi(0,\cdot)=\delta(\cdot)}.
$$
The specific expression of $\Psi(t,\cdot)$ will be given in Section \ref{HK_SEC} below. It plays the same role in the current setting as the normal density in the classical Euclidean central limit theorem.
 
\end{theorem}

\textcolor{black}{We are furthermore able to specify the previous result giving a convergence rate.}
\begin{theorem}[Local limit theorem] \label{LLT} Under \H{R}, there exists $C:=C(n,\mu_Z)$ s.t. for all $x\in \B^n$ and $N$ large enough
\begin{equation}\label{EQ_LLT}
 |f_{S_N}(x)-\Psi(t,x)|\le \frac{C}{N}.
\end{equation}
\end{theorem}

From now on we will denote by $C$ a generic constant, that may change from line to line and depends on $ n$ and the law $\mu_Z $, i.e. $C=C(n,\mu_Z)$.

\begin{remark}[About the radial densities] The radial assumption in \H{R} is somehow standard in this setting, see \cite{karp:tutu:shur:59} or Section 3.2.7 in \cite{terr:13}. It allows to focus on the asymptotic behavior of the radius to the origin for $S_N$. This isotropy makes the analysis essentially scalar. \textcolor{black}{In particular, the rotation invariance of the underlying law yields that the law of the sum does not depend on the order of the summands. It is indeed shown in \cite{ferr:ren:11} that the \textit{gyrogroup} $(\mathbb B^n,\oplus) $ is gyro-commutative, i.e. defining for $a,b,c\in \mathbb B^n$ the gyration as ${\rm gyr}[a,b]c=\ominus(a\oplus b)\oplus(a\oplus(b\oplus c)) $ then $a\oplus b={\rm gyr}[a,b](b\oplus a) $. Since the gyration ${\rm gyr} $ induces a rotation, for i.i.d. random variables $(Z^j)_{j\ge 1} $ with rotationally invariant  laws , $\bar S_N\overset{(\rm law)}{=} \oplus_{j=1}^N \frac 1N \otimes Z^{p(j)}$ (resp. $  S_N\overset{(\rm law)}{=} \oplus_{j=1}^N \frac 1{\sqrt N} \otimes Z^{p(j)}$), where $p $ is any permutation of $(1,\cdots,N) $}.
\end{remark}

\begin{remark}[About the local theorem]
The convergence rate in Theorem \ref{LLT} is due to the fact that, since we are dealing with radial densities the first moment is already zero. On the other hand one could expect that on the r.h.s. of \eqref{EQ_LLT} there should as well be an exponential decay term of the form $\exp(-cd^2(x)/t),\ d^2(x)=2\atanh(\|x\|) $ standing for the hyperbolic distance of $x$ to the origin. This would indeed correspond to the Aronson type bound for the hyperbolic heat kernel, see e.g. Molchanov \cite{molc:75} and would extend the \textit{full} local limit theorem of Batthacharya and Rao \cite{bhat:rao:76} to the current setting. One can expect this result to hold but it would require a much more involved analysis than the one in the current work.

\end{remark}

\mysection{The Poincar\'e ball: Fourier inversion and heat kernel  }
\label{SEC_FOUR}

To prove our results, as for any (local) limit theorem see e.g. \cite{bhat:rao:76} in the Euclidean case, we aim at comparing the \textit{characteristic functions} of the underlying \textit{normal} law and the one of the corresponding approximating walk.

In the current non-Euclidean setting, the \textit{characteristic function} will be expressed in terms of  the Fourier-Helgason transform, see \cite{helg:84}. 

\subsection{The Fourier-Helgason transform and its inverse}
In the following for a generic radial function $f:\B^n\rightarrow \R$ we will write with a slight abuse of notation $f(z)=f(r) $, $r=\|z\|\in [0,1) $. For  a radial function $f\in C_0^\infty(\B^n,\R) $, and $\lambda \in \R $, its Fourier-Helgason transform is given 
by the expression
\begin{equation}
\hat f(\lambda)=\int_{\B^n}f(x)\varphi_\lambda(x) \mu_{\B^n}(dx)= \Omega_{n-1}\int_0^{+\infty} (\sinh \eta)^{n-1}f( \tanh(\frac \eta 2) )\varphi_\lambda (\tanh(\frac \eta 2))d\eta,\label{TF_H}
\end{equation}
where $\Omega_{n-1}=\frac{2\pi^{\frac n2}}{\Gamma\left( \frac n2 \right)} $ denotes the area of the unit sphere $\mathbb S^{n-1} $ and the radial functions $\varphi_\lambda $ are eigen-functions of  the Laplace-Betrami operator in $\B^n $ expressed in radial coordinates.
Namely, for a generic smooth enough $\phi:\B^n\rightarrow \R $,
$$\Delta_{\B^n}\phi(z)=\Big( \frac {1-\|z\|^2}2\Big)^2\sum_{j=1}^n \partial_{z_j}^2 \phi(z)+(n-2)\frac{1-\|z\|^2}2\sum_{j=1}^n z_j\partial_{z_j}\phi(z).$$ 
When $\phi$ is radial, this operator rewrites as
\begin{equation*}
\Delta_{\B^n}\phi(r)=\partial_r^2 \phi(r)+(n-1)\coth(r)\partial_r \phi(r),
\end{equation*}
and the function $\varphi_\lambda $ in \eqref{TF_H} solves the differential equation:
 \begin{equation}
 \label{DIFF_EQ}
 \begin{cases}
 \Delta_{\B^n} \varphi_\lambda(r)+(\lambda^2+\rho^2) \varphi_\lambda(r)=0,\ \rho=\frac {n-1}2,\\
\varphi_\lambda(0)=1.
 \end{cases}
 \end{equation}
 This equation is solved using the functions
\begin{align}
e_{\lambda,\omega}(x)=\frac{(1-\|x\|^2)^{\frac 12(n-1+\textcolor{black}{i\lambda})}}{\|x-\omega\|^{n-1+\textcolor{black}{i\lambda}}},\label{DEF_Z}
\end{align}
which are  eigenfunctions of $\Delta_{\B^n} $ associated with the eigenvalue $-(\lambda^2+\rho^2) $ and therefore. Those functions actually play in the current context a similar role to the complex exponential $\exp(i x\cdot y) $ in the usual Fourier analysis on the Euclidean space $\R^n$.

We then define the corresponding elementary spherical function $\varphi_\lambda $ setting for all $x\in \B^n $,
\begin{align}
\label{EIG_SPHERE_HARMO}
\varphi_\lambda(x)=\frac{1}{\Omega_{n-1}}\int_{\S^{n-1}}e_{\lambda,\omega}(x)\Lambda (d\omega),
\end{align}
where $\Lambda $ stands for the Lebesgue measure on the unit sphere $\S^{n-1} $. It is then clear that $\varphi_\lambda $ is also an eigenfunction of $\Delta_{\B^n} $ with eigenvalue $-(\lambda^2+\rho^2) $ and that $\varphi_\lambda(0)=1 $. Hence, the above $\varphi_\lambda $ indeed solves \eqref{DIFF_EQ}.

 Note that its expression also writes in geodesic polar coordinates (see e.g. \cite{anke:17}):
 \begin{align}\label{EIG_FUNC_GEO}
 \varphi_\lambda (\tanh(\frac \eta 2))=
 \tilde C_n\sinh(\eta)^{2-n}\int_0^\eta ds(\cosh(\eta)-\cosh(s))^{\frac {n-3}2} \cos(\lambda s),\ 
\tilde C_n =
\frac{2^{\frac {n-1}2}\Gamma\left(\frac n2\right)}{\sqrt \pi \Gamma\left(\frac{n-1}2\right)}.
 \end{align}
 For computational purposes we also make the following link with the Legendre special function, see Gradstein and Ryzhik \cite{grad:ryzh:07}, p. 1016, formula 8.715(1),
 \begin{align}
 P_{-\frac 12+i\lambda}^{1-\frac n2}(\cosh(\eta))&=\sqrt{\frac 2\pi}\frac{\sinh(\eta)^{1-\frac n2}}{\Gamma(\rho)}\int_0^\eta ds(\cosh(\eta)-\cosh(s))^{\frac {n-3}2} \cos(\lambda s),\notag\\
 \varphi_\lambda (\tanh(\frac \eta 2))&=2^{\rho-\frac 12}\Gamma(\rho+\frac 12)\sinh(\eta)^{\frac 12-\rho}P_{-\frac 12+i\lambda}^{\frac 12-\rho}(\cosh(\eta)). \label{CORRESP_PHI_LEGENDRE}
 \end{align}

 The inversion formula associated with \eqref{TF_H}, for a radial function $f\in C_0^\infty(\B^n,\R)$, reads for all $\|z\|=r\in [0,1) $,
 \begin{equation}
\label{INV}
f(z)=f(r)= C_n\int_{0}^{+\infty} d\lambda |\c(\lambda)|^{-2}\hat f(\lambda ) \varphi_\lambda (r),
 \end{equation}
 where $\c(\lambda) $ is the generalized Harish-Chandra function:
 \begin{equation}\label{HC}
 \c(\lambda)=\frac{2^{3-n-i\lambda}\Gamma(\frac n2)\Gamma(i\lambda)}{\Gamma\left( \frac{n-1+i\lambda}2\right)\Gamma\left( \frac{1+i\lambda}2\right)},
 \end{equation}
 and $C_n=\frac1{2^{n-3} \pi\Omega_{n-1}} $. Formulas \eqref{TF_H} and \eqref{INV} can be derived from Liu and Peng \cite{liu:peng:09} (taking therein $\vartheta=2-n $)
 or Ferreira \cite{ferr:15} (taking therein $\sigma=2-n, t=1$). 
 \subsection{The heat kernel on $\B^n$.}\label{HK_SEC}
 The heat kernel will provide the limit law which is somehow the analogue in the current non-Euclidean setting of the normal law. The normal density of parameter $t>0$ in $\B^n $ is defined as the solution to
 \begin{equation*}
\frac 12 \Delta_{\B^n}\Psi (t,x)=\partial_t \Psi (t,x), \ \textcolor{black}{\Psi(0,\cdot)=\delta(\cdot)}.
 \end{equation*}
 In the literature the usual heat equation considered is
  \begin{equation*}
 \Delta_{\B^n}\psi (t,x)=\partial_t \psi (t,x), \ \textcolor{black}{\psi(0,\cdot)=\delta(\cdot)}.
 \end{equation*}
 It can actually be solved through the Fourier-Helgason transform (assuming the solution is radial). Namely, from \eqref{TF_H}, one derives that for all $\lambda\in \R $:
 \begin{equation}\label{TF_HK_1}
\hat \psi(t,\lambda)=\exp(-(\lambda^2+\rho^2)t), \ \rho=\frac{n-1}2.
 \end{equation}
 Observe that the expression in \eqref{TF_HK_1} in particular coincides when $n=2,3 $ with the one obtained in \cite{karp:tutu:shur:59}
 . 
 With the terminology of this work the expression in \eqref{TF_HK_1} corresponds to the so-called characteristic functions of the first kind. In order to match the Euclidean probabilistic set-up we will consider the characteristic functions of the second kind,
 \begin{equation}\label{TF_HK_2}
\hat \psi_2(t,\lambda):=\frac{\hat \psi(t,\lambda)}{\hat \psi(t,0)}=\exp(-\lambda^2t),
 \end{equation} 
 so that in particular $\hat \psi_2(t,0)=1 $. By inversion it follows that for $m\in \N$:
 \begin{equation}\label{HK_H}
\psi(t,x)=\psi(t,\atanh(\frac \eta2))=\begin{cases}\frac{\exp(-\frac{m^2 t}{2})}{(2\pi)^m\sqrt{2\pi t }} \Big(-\frac 1{\sinh \eta} \partial_\eta \Big)^m\exp\left(-\frac{\eta^2}{2t} \right),\ n=2m+1,\\
\frac{\exp(-\frac{(m-\frac 12)^2 t}{2})}{(2\pi)^m\sqrt{\pi t }}\int_{\eta}^{+\infty} \frac{ds}{\sqrt{\cosh(s)-\cosh(\eta)}}(-\partial_s)\Big(-\frac 1{\sinh s} \partial_s \Big)^{m-1}\exp\left(-\frac{s^2}{2t} \right),\ n=2m.
\end{cases}
 \end{equation}
 Namely, the heat kernel on $ \mathbb H^n$ is expressed in term of the hyperbolic distance to the origin. We can refer to \cite{anke:17} for the derivation of \eqref{HK_H} through the Abel transform and its inverse or to \cite{grig:nogu:97} in which the authors derive the heat kernel through the fundamental solution of the wave equation.
 
 From the previous definitions it is clear that for $\Psi $ as in Theorem \ref{CLT}, $\Psi(t,x)=\psi(\frac t2,x) $.
 
 \subsection{Some additional tools from harmonic analysis on $\B^n$}
 To establish our main theorems we will need some additional tools. Namely, we need  to define the convolution on $\B^n$.
 To this end, following Ahlfors \cite{ahlf:81} we define for fixed $a\in \B^n$ the translation operator,
 \begin{equation*}
T_a: x\in \B^n \mapsto T_a(x)=-a\oplus x\in \B^n.
 \end{equation*}
 The previous mapping is bijective, it is easily checked that $T_a^{-1}=T_{-a} $, and has the next important properties.
 \begin{proposition}[Some properties of $T_a $, cf. \cite{ahlf:81}, pp. 18--30]\label{PROP_TA}
 For a fixed $a$ and $x\in \B^n$ it holds that:
 \begin{trivlist}
\item[1.] $$1-\|T_a(x)\|^2 =\frac{(1-\|a\|^2)(1-\|x\|^2)}{1-2\langle x,a\rangle+\|x\|^2\|a\|^2}.$$
\item[2.] $$ \det(D_xT_a(x))=\Bigg(\frac{1-\|a\|^2}{1-2\langle x,a\rangle+\|x\|^2\|a\|^2}\Bigg)^n,\ \vvvert D_xT_a(x)\vvvert=\frac{1-\|a\|^2}{1-2\langle x,a\rangle+\|x\|^2\|a\|^2},$$
where $\vvvert\cdot \vvvert $ denotes the spectral norm.
\item[3.] $T_a$ preserves the Riemann measure. Namely,
$$\frac{\vvvert D_xT_a(x)\vvvert}{1-\|T_a(x)\|^2} =\frac{1}{1-\|x\|^2}.$$ 
\item[4.] $T_a(x)=-\frac{D_x T_a(x)}{\vvvert D_xT_a(x)\vvvert} T_x(a) $. In particular, for a radial function $f:\B^n\rightarrow \R $, it holds that
$$f(T_a(x))=f(T_x(a)).$$
 \end{trivlist}
 \end{proposition}
 The translation operator allows to define in a quite natural way the convolution. Namely, for  $f,g\in C_0^\infty(\B^n,\R) $, we set:
 \begin{equation}\label{DEF_CONV}
 \forall x\in \B^n,\ f*g(x)=\int_{\B^n}f(-y\oplus x)g(y)\mu_{\B^n}(dy)=\int_{\B^n}f( T_y(x))g(y)\mu_{\B^n}(dy).
 \end{equation}
This definition enlarges the one in \cite{karp:tutu:shur:59} to the current multi-dimensional setting.
 We now state some very useful properties of the convolution in order to prove limit theorems
  \begin{proposition}[Some properties of the convolution]\label{PROP_CONV}
Let   $f,g\in C_0^\infty(\B^n,\R) $ and be radial functions. The following properties hold:
\begin{trivlist}
\item[-] Commutativity: for the definition in \eqref{DEF_CONV} it holds that 
$$f*g(x)=g*f(x). $$
\item[-] Stability of the radial property through convolution :  $f*g $ is a radial function. 

\item[-] Fourier-Helgason transform of the convolution:  it holds with  the definition in \eqref{TF_H} that for all $\lambda\in \R $,
$$\widehat {f*g}(\lambda)=\hat f(\lambda) \hat g (\lambda).$$
\end{trivlist}
 \end{proposition}
 \begin{proof}Since $f$ is radial, we have from point 4 of Proposition \ref{PROP_TA} that $f(T_y(x))=f(T_x(y)) $.
   Setting $z=T_x(y) =-x\oplus y\iff x\oplus z =T_{-x}(z)=y$ and using point 3 of Proposition \ref{PROP_TA},
   we derive:
 \begin{align*}
f*g(x)=&\int_{\B^n}f(T_y(x))g(y)\frac{dy}{(1-\|y\|^2)^n}=\int_{\B^n}f(z)g(T_{-x}(z))\frac{dz}{(1-\|z\|^2)^n}\\
\underset{x\oplus z=-[(-x)\oplus (-z)]}{=}&\int_{\B^n}f(z)g(T_{x}(-z))\frac{dz}{(1-\|z\|^2)^n}\\
\underset{\tilde z=-z}{=}&\int_{\B^n}f(\tilde z)g(T_{x}(\tilde z))\frac{d\tilde z}{(1-\|\tilde z\|^2)^n}=\int_{\B^n}f(\tilde z)g(T_{\tilde z}(x))\frac{d\tilde z}{(1-\|\tilde z\|^2)^n}
=g*f(x),
 \end{align*}
where we also used that \textcolor{black}{$f,g$ are} radial (and therefore even), for the equality in the second \textcolor{black}{and third lines}.
 This proves the first point.
 
 Let us now turn to the second point. It suffices to prove that for any orthogonal matrix $A\in O(n)$ and $x\in \B^n$, it actually holds that
 \begin{align*}
 f*g(Ax)=f(x).
 \end{align*}
 Observe first from \eqref{Mobius_ADD} that for all $a,b\in \B^n $
 \begin{align}
 \label{ORTHO_IN_MOBIUS} A(A^T \textcolor{black}{a\oplus b})=a\oplus Ab.
 \end{align}
Indeed,
\begin{align*}
A(A^T a\oplus b)=&A\frac{(1+2\langle A^Ta,b\rangle+\|b\|^2)A^Ta+(1-\|A^Ta\|^2)b}{1+2\langle A^Ta,b\rangle+\|A^Ta\|^2\|b\|^2}=\frac{(1+2\langle a,Ab\rangle+\|Ab\|^2)a+(1-\|a\|^2)Ab}{1+2\langle a,Ab\rangle+\|a\|^2\|Ab\|^2}\\
=&	a\oplus Ab.
\end{align*}
Write now,
\begin{align*}
f*g(Ax)=&\int_{\B^n} f(y)g(-y\oplus Ax) \mu_{\B^n}(dy)\underset{\eqref{ORTHO_IN_MOBIUS}}{=}\int_{\B^n}f(y)g(A(A^T(-y)\oplus x))\mu_{\B^n}(dy)\\
=&\int_{\B^n}f(y)g(-(A^Ty)\oplus x))\mu_{\B^n}(dy)=\int_{\B^n}f(y)g(T_{A^Ty}(x))\mu_{\B^n}(dy)=\int_{\B^n}f(y)g(T_{x}(A^Ty))\mu_{\B^n}(dy)\\
=&\int_{\B^n}f(Ay)g(T_{x}(y))\mu_{\B^n}(dy)=\int_{\B^n}f(y)g(T_{y}(x))\mu_{\B^n}(dy)=f*g(x),
\end{align*}
where we used that $f,g$ are radial \textcolor{black}{and property 3 of Proposition \ref{PROP_TA}}. This proves the second point.

Let us eventually prove the property concerning the Fourier-Helgason transform of the convolution. From \eqref{TF_H} write for all $\lambda\in \R $:
\begin{align*}
\widehat {f*g}(\lambda)=&\int_{\B^n}\Big(\int_{\B^n} f(T_y(x)) g(y)\mu_{\B^n}(dy) \Big) \varphi_{-\lambda}(x)\mu_{\B^n}(dx) =\int_{\B^n} g(y)\Bigg(\int_{\B^n}f(T_y(x))\varphi_{-\lambda}(x)\mu_{\B^n}(dx)\Bigg) \mu_{\B^n}(dy)\\
=&\int_{\B^n} g(y)\Bigg(\int_{\B^n}f(z)\varphi_{-\lambda}(T_{-y}(z))\mu_{\B^n}(dz)\Bigg) \mu_{\B^n}(dy).
\end{align*}
Going back to the definition \eqref{EIG_SPHERE_HARMO} and \eqref{DEF_Z} we derive:
\begin{align*}
\varphi_{-\lambda}(T_{-y}(z))=\frac{1}{\Omega_{n-1}}\int_{\S^{n-1}} e_{-\lambda,\omega}(T_{-y}(z))\Lambda(d\omega).
\end{align*}
Now, it holds from Lemma 4.3 in \cite{liu:peng:09} that 
$$e_{-\lambda,\omega}(T_{-y}(z))=e_{-\lambda,\omega}(-y)e_{-\lambda,T_{-y}(\omega)}(z). $$
We point out that for $\omega\in \S^{n-1} $ it indeed holds from \textcolor{black}{Property 1. of Proposition \ref{PROP_TA}} that $T_{-y}(\omega)\in \S^{n-1} $.

Thus,
\begin{align}
\widehat {f*g}(\lambda)&=\int_{\B^n} g(y)\Bigg(\int_{\B^n}f(z)\frac{1}{\Omega_{n-1}}\int_{\S^{n-1}} e_{-\lambda, \omega}(-y)e_{-\lambda,T_{-y}(\omega)}(z)\Lambda(d\omega)\mu_{\B^n}(dz)\Bigg) \mu_{\B^n}(dy)\notag\\
&=\frac{1}{\Omega_{n-1}}\int_{\S^{n-1}}\int_{\B^n} g(y)e_{-\lambda, \omega}(-y)\left(\int_{\B^n} f(z) e_{-\lambda,T_{-y}(\omega)}(z)\mu_{\B^n}(dz)\right) \mu_{\B^n}(dy)\Lambda(d\omega).\label{PREAL_PC}
\end{align}
We now recall that for possibly non-radial functions, $h\in C_0^\infty(\B^n)$ the \textit{general} Fourier-Helgason transform has two arguments (the polar coordinates of the phase variable) and writes for  $\lambda\in \R,\ \zeta \in \S^{n-1} $:
$$\hat h(\lambda,\zeta)=\int_{\B^n} h(x) e_{-\lambda, \zeta}(x)\mu_{\B^n}(dx).$$
When $h$ is radial we derive passing to polar coordinates that  $\hat h(\lambda,\zeta)=\hat h(\lambda)$ so that
$$ \Omega_{n-1 }\hat h(\lambda)=\int_{\B^n} h(x) \int_{\S^{n-1}}e_{-\lambda, \zeta}(x)\Lambda(d\zeta)\mu_{\B^n}(dx),$$
from which we derive 
$$ \hat h(\lambda)=\int_{\B^n} h(x) e_{-\lambda, \zeta}(x)\mu_{\B^n}(dx), \forall \zeta \in \S^{n-1},$$
i.e. the Fourier transform does not depend on the considered point \textcolor{black}{on the sphere $\S^{n-1}$}. Taking $h=f$, we therefore deduce from \eqref{PREAL_PC} that
$$\widehat {f*g}(\lambda)=\frac{1}{\Omega_{n-1}}\int_{\S^{n-1}}\int_{\B^n} g(y)e_{-\lambda, \omega}(-y) \hat f(\lambda) \mu_{\B^n}(dy)\Lambda(d\omega)=\hat g(\lambda) \hat f(\lambda),$$
which gives the third point and concludes the proof of the proposition.
 \end{proof}

  \subsection{Non Euclidean mean, variance ans scaling}
We define, coherently with the Euclidean case, the mean and variance associated with a $\B^n $-valued random variable $Z$ defined on some probability space $(\Omega,\F,\P)$ satisfying assumption \H{R}.

The \textit{analogue} of the characteristic function (of the second kind with the terminology of \cite{karp:tutu:shur:59}) writes:
\begin{equation}
\label{CAR_FUNC}
\forall \lambda \in \R,\ \Phi_Z(\lambda)=\frac{\hat f_Z(\lambda)}{\hat f_Z(0)}.
\end{equation}
Again the normalization guarantees that $\Phi_Z(0)=1 $ as in the Euclidean case. We also emphasize that, since we assumed the density to be radial, it follows from \eqref{TF_H} and \eqref{EIG_FUNC_GEO} that for all $m=2j+1, j\in \N$,
\begin{equation}\label{0_ODD_MOMENTS}
\partial_\lambda^m \Phi_Z(\lambda)|_{\lambda=0}=0.
 \end{equation}
 In other terms, the \textit{odd} moments of the random variable are 0.

From the above definition we then define the \textit{analogue} of the variance as
\begin{equation}\label{DEF_VAR}
V_Z=-\partial_\lambda^2 \Phi_Z(\lambda)|_{\lambda=0}.
\end{equation}
\begin{proposition}[Variance of the sum]
Let $Z_1$ and $Z_2$ be two $\B^n $-valued independent random variables with radial densities $f_{Z_1},f_{Z_2}\in C_0^\infty(\B^n,\R)$ w.r.t. the Riemannian volume \eqref{VOLUME} of $\B^n$. It then holds that
$$
V_{Z_1\oplus Z_2}=V_{Z_1}+V_{Z_2}.$$
\end{proposition}
\begin{proof}
It is clear from \eqref{DEF_CONV} and Proposition \ref{PROP_CONV} that $Z_1\oplus Z_2$ has a radial density $f_{Z_1\oplus Z_2} $ w.r.t. $\mu_{\B^n} $ and that
$$f_{Z_1\oplus Z_2}=f_{Z_1}*f_{Z_2},\ \hat f_{Z_1\oplus Z_2}=\hat f_{Z_1} \hat f_{Z_2}.$$
Hence, from \eqref{0_ODD_MOMENTS}-\eqref{DEF_VAR},
\begin{align*}
V_{Z_1\oplus Z_2}=&\frac{-\partial_\lambda^2 \Big(\hat f_{Z_1}(\lambda)\Big)|_{\lambda=0}\hat f_{Z_2}(0)- 2\Big(\partial_\lambda\hat f_{Z_1}(\lambda)\partial_\lambda\hat f_{Z_2}(\lambda)\Big)|_{\lambda=0}- \hat f_{Z_1}(0)\partial_\lambda^2\Big(\hat f_{Z_2}(\lambda)\Big)|_{\lambda=0}}{\hat f_{Z_1}(0)\hat f_{Z_2}(0)}\\
&=-\frac{\partial_\lambda^2 \Big(\hat f_{Z_1}(\lambda)\Big)_{\lambda=0}}{\hat f_{Z_1}(0)}-\frac{\partial_\lambda^2 \Big(\hat f_{Z_2}(\lambda)\Big)_{\lambda=0}}{\hat f_{Z_2}(0)}=V_{Z_1}+V_{Z_2}.
\end{align*}

\end{proof}

In order to make a connection with the usual normalizing rate in the CLT, in \cite{karp:tutu:shur:59} the authors introduce the following invariance property on the radial measures.  Let $Z$ be a $\B^n $-valued random variable
with density $f_Z\in C_0^\infty(\B^n) $ and denote again by $\mu_{Z,R} $ the induced measure in geodesic coordinates. For a parameter $\varepsilon>0 $, the authors consider scaled variables $ Z_\varepsilon $ \textcolor{black}{(associated with $Z$)} for which they assume that the corresponding measure $\mu_{ Z_\varepsilon, R} $ \textcolor{black}{in radial geodesic coordinates} satisfies for all $\tau>0 $:
\begin{align}\label{SCALE_KTS}
\mu_{ Z_\varepsilon,R}\Big(\Big\{ \eta \in \R_+: \eta\le \tau\varepsilon\Big\} \Big)=\mu_{ Z,R}\Big(\Big\{ \eta \in \R_+: \eta\le \tau\Big\} \Big).
\end{align}

\begin{proposition}[Scaled Variables]\label{SCALED_KTS}
In order to fulfill property \eqref{SCALE_KTS}, the density $f_{ Z_\varepsilon} $ must be chosen such that, in geodesic polar coordinates:
\begin{equation}\label{SCALED_DENS_PROP}
\forall \eta\in (0,+\infty),\ f_{ Z_\varepsilon}(\tanh(\frac \eta 2))=\frac 1\varepsilon f_Z(\tanh(\frac{\eta}{2\varepsilon}))\Big(\frac{\sinh( \frac {\eta}{\varepsilon})}{\sinh ( \eta)}\Big)^{n-1}.
\end{equation}
Observe in particular that if ${\rm supp}(f_Z)\in B(0,R) $ then ${\rm supp }(f_{ Z_\varepsilon})\in B(0,R\varepsilon) $. Furthermore, for a bounded measurable function $g$ (not necessarily radial):
\begin{align}
\E[g( Z_\varepsilon)]&
=\E\Big[g\Big(\tanh(\varepsilon \atanh(\|Z\|) ) \frac{Z}{\|Z\|} \Big)\Big].\label{FUNC_MEAN_SCALED}
\end{align}

Importantly, from the definition of the M\"obius multiplication in \eqref{MULT_MOBIUS}, it holds that the  property \eqref{SCALE_KTS} holds if and only if:
$$ Z_\varepsilon \overset{({\rm law})}{=} \varepsilon \otimes Z.$$

\end{proposition}

\begin{proof}
Since we assumed $Z $ to have a smooth compacted support density in $\B^n $, the above identity \textcolor{black}{\eqref{SCALED_DENS_PROP}} gives that this is also the case for $ Z_\varepsilon $. Denoting the corresponding density w.r.t. $\mu_{\B^n} $ by $f_{ Z_\varepsilon} $, we write: 
\begin{align*}
\mu_{ Z_\varepsilon,R}\Big(\Big\{ \eta \in \R_+: \eta\le \tau\varepsilon\Big\} \Big)&=\int_{0}^{ \tau \varepsilon} \mu_{ Z_{\varepsilon},R}(d\eta)=\Omega_{n-1}\int_{0}^{\tau \varepsilon} f_{ Z_{\varepsilon}}(\tanh(\frac \eta 2)) (\sinh \eta)^{n-1} d\eta\\
&\underset{\tilde \eta=\frac  \eta \varepsilon}{=}\Omega_{n-1}\int_0^\tau f_{ Z_{\varepsilon}}(\tanh(\frac {\varepsilon\tilde \eta}{ 2})) (\sinh(\varepsilon\tilde \eta))^{n-1}\varepsilon d\tilde \eta\\
&=\Omega_{n-1}\int_0^\tau \Big[\varepsilon f_{ Z_{\varepsilon}}(\tanh(\frac {\varepsilon\tilde \eta}{ 2})) \Big(\frac{\sinh(\varepsilon \tilde \eta)}{\sinh (\tilde \eta)}\Big)^{n-1}\Big]\sinh(\tilde \eta)^{n-1}d\eta.
\end{align*}
The scaling identity \eqref{SCALE_KTS} then holds if and only if
\begin{equation*}
 \Big[\varepsilon f_{ Z_{\varepsilon}}(\tanh(\frac {\varepsilon \tilde \eta }{ 2})) \Big(\frac{\sinh(\varepsilon\tilde \eta)}{\sinh (\tilde \eta)}\Big)^{n-1}\Big]=f_{Z}(\tanh(\frac{\tilde \eta}2))\iff f_{ Z_\varepsilon}(\tanh(\frac \eta 2))=\frac 1\varepsilon f_Z(\tanh(\frac{\eta}{2\varepsilon}))\Big(\frac{\sinh( \frac {\eta}{\varepsilon})}{\sinh ( \eta)}\Big)^{n-1}.
 \end{equation*}
Importantly, since $\supp(f_Z) \subset B(0,R) $ for some $R\in (0,1)$ it then holds that $\supp(f_{ Z_\varepsilon})\in (0,\textcolor{black}{\tanh(\varepsilon \atanh(R))})$.

Let us specify what is the \textit{transform} of the initial random variable $Z$ which has density $f_{ Z_\varepsilon}$. For a bounded measurable function $g$ and \textcolor{black}{from \eqref{RADIAL_HYPER}-\eqref{RADIAL_GE0_MEAS} write}:
\begin{align*}
\E[g( Z_\varepsilon)]&=\int_{\B^n} g(z) f_{ Z_\varepsilon}(z) \mu_{\B^n}(dz)\\
&=\int_{\mathbb S^{d-1}}\Lambda (d\Theta)\int_0^{+\infty} g(\tanh(\frac{\tilde \eta}2)\Theta) f_{ Z_\varepsilon}(\tanh(\frac{\tilde \eta}2) )\sinh(\tilde \eta)^{n-1}d \tilde \eta\\
&=\int_{\mathbb S^{d-1}}\Lambda (d\Theta)\int_0^{+\infty} g(\tanh(\frac{\tilde \eta}2)\Theta)\frac 1\varepsilon f_Z(\tanh(\frac{\tilde \eta}{2\varepsilon}))\Big(\frac{\sinh( \frac {\tilde \eta}{\varepsilon})}{\sinh ( \tilde \eta)}\Big)^{n-1}\sinh(\tilde \eta)^{n-1}d\tilde \eta\\
&\underset{\eta=\frac{\tilde \eta}\varepsilon}{=}\int_{\mathbb S^{d-1}}\Lambda (d\Theta)\int_0^{+\infty}g(\tanh(\frac{\varepsilon \eta}{2})\Theta)f_Z(\tanh(\frac{\eta}{2}))\sinh(\eta)^{n-1}d \eta\\
&=\int_{\mathbb S^{d-1}}\Lambda (d\Theta)\int_0^{+\infty}g\Big(\tanh(\frac{\varepsilon 2 \atanh(\tanh(\frac \eta 2))}{2})\Theta\Big)f_Z(\tanh(\frac{\eta}{2}))\sinh(\eta)^{n-1}d \eta\\
&=\int_{\B^n}g\Big(\tanh(\varepsilon \atanh (\|z\|))\frac{z}{\|z\|} \Big)f_Z(z) \mu_{\B^n}(dz)=\E\Big[g\Big(\tanh(\varepsilon \atanh(\|Z\|))\frac{Z}{\|Z\|}\Big)\Big].
\end{align*}
\end{proof}
Hence, to verify the condition \eqref{SCALE_KTS} we have to take $ Z_\varepsilon \overset{({\rm law})}{=}\tanh(\varepsilon \atanh(\|Z\|)\frac{Z}{\|Z\|})\textcolor{black}{=\varepsilon \otimes Z}\neq \varepsilon Z $, that would \textit{somehow} be the natural scaling in the Euclidean setting.

\begin{proposition}[Variance for the random walks]\label{VAR_OF_THE_WALKS}
Let $Z$ satisfy \H{R}.  Set for $\varepsilon>0 $, $ Z_\varepsilon := \varepsilon \otimes Z$. It then holds that there exists $C\ge 1$ s.t.
\begin{equation}\label{VAR_CTR}
V_{ Z_\varepsilon}\le C\varepsilon^2.
\end{equation} 
In particular, choosing $\varepsilon=\frac {1}{\sqrt N} $, the above control can be specified to derive that with the notation of \eqref{WALKS}:
\begin{equation}\label{VAR_LIM}
V_{S_N}=V_{\oplus_{j=1}^N \frac 1{\sqrt N}\otimes Z^j}\underset{N}{\longrightarrow} \textcolor{black}{t :=\frac 1n \int_0^{+\infty } \tilde \eta ^2  \mu_{Z,R}(d\tilde \eta)},
\end{equation}
which is precisely the asymptotic variance appearing in Theorems \ref{CLT} and \ref{LLT}.
Furthermore, there exists $C\ge 1$ s.t.
\begin{equation}\label{CONV_RATE_FOR_VARIANCE}
\textcolor{black}{|V_{S_N}-t|\le \frac{C}{N}},
\end{equation}
and 
\begin{equation}\label{VAR_LIM_LLN}
V_{\bar S_N}=V_{\oplus_{j=1}^N \frac 1{ N}\otimes Z^j}\underset{N}{\longrightarrow} 0,
\end{equation}
which is also coherent with the statement of Theorem \ref{LLN}.
\end{proposition}
\begin{proof}
For $ Z_\varepsilon  $, we again derive from \eqref{DEF_VAR}, \eqref{EIG_FUNC_GEO} that 
\begin{align}
V_{ Z_\varepsilon}=&-\partial_\lambda^2 \frac{\hat f_{ Z_\varepsilon}(\lambda)}{\hat f_{ Z_\varepsilon}(0)}\Big|_{\lambda=0}=-\partial_\lambda^2\frac{\Omega_{n-1}}{\hat f_{Z_\varepsilon}(0)}  \int_0^{+\infty}f_{ Z_\varepsilon}(\tanh(\frac{\eta}{2}))( \varphi_\lambda(\tanh(\frac \eta 2))) \sinh(\eta)^{n-1}d\eta\Big|_{\lambda=0}\notag\\
=& -\partial_\lambda^2\frac{\Omega_{n-1}}{\hat f_{Z_\varepsilon}(0)}\int_0^{+\infty}\ f_Z(\tanh(\frac{\tilde \eta}2))\varphi_\lambda(\tanh(\frac {\varepsilon \tilde \eta}2))\sinh (\tilde \eta)^{n-1} d\tilde \eta\Big|_{\lambda=0}\notag\\
=&\frac{1}{\hat f_{ Z_\varepsilon}(0)} \int_0^{+\infty} \Big( \tilde C_n\sinh(\bar \eta)^{2-n}\int_0^{\bar \eta} ds(\cosh(\bar \eta)-\cosh(s))^{\frac {n-3}2} s^2 \cos(\lambda s) \Big)\Big|_{\bar \eta=\varepsilon \tilde \eta}\mu_{Z,R}(d\tilde \eta)\Big|_{\lambda=0}\label{EQ_VAR_EPS}\\
=&\frac{1}{\hat f_{ Z_\varepsilon}(0)} \int_0^{+\infty} \Big( \tilde C_n\sinh(\bar \eta)^{2-n}\int_0^{\bar \eta} ds(\cosh(\bar \eta)-\cosh(s))^{\frac {n-3}2} s^2 \Big)\Big|_{\bar \eta=\varepsilon \tilde \eta}\mu_{Z,R}(d\tilde \eta)\notag\\
\le & \frac{1}{\hat f_{ Z_\varepsilon}(0)} \int_0^{+\infty} \Big( \tilde C_n \bar \eta^2\sinh(\bar \eta)^{2-n}\int_0^{\bar \eta} ds(\cosh(\bar \eta)-\cosh(s))^{\frac {n-3}2}  \Big)\Big|_{\bar \eta=\varepsilon \tilde \eta}\mu_{Z,R}(d\tilde \eta)\notag\\
\le &\frac{\varepsilon^2}{\hat f_{ Z_\varepsilon}(0)} \int_0^{+\infty} \tilde  \eta ^2\varphi_{0}( \tanh(\varepsilon\frac{\tilde  \eta} 2))\mu_{Z,R}(d\tilde \eta)\le \textcolor{black}{C \frac{\varepsilon^2}{\hat f_{\bar Z_\varepsilon}(0)} \int_{0}^\infty\tilde  \varphi_{0}( \tanh(\varepsilon\frac{\tilde  \eta} 2))\mu_{Z,R}(d\tilde \eta) \le C\varepsilon^2},\notag
\end{align}
\textcolor{black}{using that $\mu_{Z,R}$ has compact support in $\R_+ $ for the last but one inequality and recalling that $\int_{0}^\infty\tilde  \varphi_{0}( \tanh(\varepsilon\frac{\tilde  \eta} 2))\mu_{Z,R}(d\tilde \eta)=\hat f_{Z_\varepsilon}(0) $ for the last one}.

In particular, considering now $S_{N}=\oplus_{j=1}^n \frac{1}{\sqrt N}\otimes Z^i$ where the $( Z^i)_{i\ge 1}$ are independent with the same law as $Z$, there exists $C\ge 1$ s.t. for all $N\ge 1$,
$$V_{S_N}=V_{\oplus_{j=1}^N \frac 1{\sqrt N}\otimes Z^{i}}=N V_{\frac {1}{\sqrt N}\otimes Z}\le C.$$
Let us now establish \eqref{VAR_LIM}. To this end we must precisely expand the inner integral in \eqref{EQ_VAR_EPS}. 
To this end we introduce the quantity:
\begin{align*}
I_{\varepsilon}(\tilde \eta):=&\sinh(\varepsilon \tilde  \eta)^{2-n}\int_0^{\varepsilon \tilde \eta} ds(\cosh(\varepsilon \tilde \eta)-\cosh(s))^{\frac {n-3}2} s^2 \\
=&(\varepsilon \tilde \eta)^3\sinh(\varepsilon \tilde \eta)^{2-n}\int_0^{1} d\tilde s\big(2\sinh(\frac{\varepsilon \tilde \eta}2(1+\tilde s))\sinh(\frac{\varepsilon \tilde \eta}2(1-\tilde s))\big)^{\frac {n-3}2} \tilde s^2\\
=&(\varepsilon \tilde \eta)^3 (\varepsilon \tilde \eta+(\varepsilon \tilde \eta)^3+o(\varepsilon^3))^{2-n}\\
&\times \int_0^1 d\tilde s\big(2(\frac{\varepsilon \tilde \eta}2(1+\tilde s)+(\frac{\varepsilon \tilde \eta}2(1+\tilde s))^3+o(\varepsilon^3))(\frac{\varepsilon \tilde \eta}2(1-\tilde s)+(\frac{\varepsilon \tilde \eta}2(1-\tilde s))^3+o(\varepsilon^3)) \big)^{\frac{n-3}2} \tilde s^2\\
=&(\varepsilon \tilde \eta)^{3+2-n} (1+(\varepsilon \tilde \eta)^2+o(\varepsilon^2))^{2-n}2^{\frac {3-n}2}(\varepsilon \tilde \eta)^{n-3}\\
&\times \int_0^1 d\tilde s \tilde s^2\big((1+\tilde s)(1-\tilde s)\big)^{\frac{n-3}{2}}\big((1+(\frac{\varepsilon \tilde \eta}2)^2(1+\tilde s)^2+o(\varepsilon^2))(1+(\frac{\varepsilon \tilde \eta}2(1-\tilde s))^2+o(\varepsilon^2)) \big)^{\frac{n-3}2} \\
=&(\varepsilon \tilde \eta)^2(1+(\varepsilon \tilde \eta)^2+o(\varepsilon^2))^{2-n}2^{\frac {3-n}2}\\
&\times \int_0^1 d\tilde s \tilde s^2\big((1+\tilde s)(1-\tilde s)\big)^{\frac{n-3}{2}}\big((1+(\frac{\varepsilon \tilde \eta}2)^2(1+\tilde s)^2+o(\varepsilon^2))(1+(\frac{\varepsilon \tilde \eta}2(1-\tilde s))^2+o(\varepsilon^2)) \big)^{\frac{n-3}2} .
\end{align*}
It therefore follows that there exists a constant $\bar c:=\bar c(n,f_Z),$ s.t. for all $\tilde \eta \in (0,2\atanh(R)) $
$$ |I_{\varepsilon}(\tilde \eta)(\varepsilon \tilde \eta)^{-2} -\mathfrak c_n|\le \bar c \varepsilon^2, \bar c:=\bar c(n,f_Z).$$
\textcolor{black}{Observe now that, similarly to the previous computations:
\begin{align}
\hat f_{Z_\varepsilon}(0)&=\tilde C_n\int_0^{+\infty} J_{\varepsilon}(\tilde \eta)\mu_{Z,R}(d\tilde \eta),\label{FOR_ASYMP_EPS_TF}\\
J_{\varepsilon}(\tilde \eta)&=\sinh(\varepsilon \tilde  \eta)^{2-n}\int_0^{\varepsilon \tilde \eta} ds(\cosh(\varepsilon \tilde \eta)-\cosh(s))^{\frac {n-3}2}  \notag\\
=&(1+(\varepsilon \tilde \eta)^2+o(\varepsilon^2))^{2-n}2^{\frac {3-n}2}\notag\\
&\times \int_0^1 d\tilde s \big((1+\tilde s)(1-\tilde s)\big)^{\frac{n-3}{2}}\big((1+(\frac{\varepsilon \tilde \eta}2)^2(1+\tilde s)^2+o(\varepsilon^2))(1+(\frac{\varepsilon \tilde \eta}2(1-\tilde s))^2+o(\varepsilon^2)) \big)^{\frac{n-3}2} .\notag
\end{align}
We now rewrite from \eqref{EQ_VAR_EPS} and \eqref{FOR_ASYMP_EPS_TF},
\begin{align*}
V_{Z_\varepsilon}&=\frac{\int_{0}^{+\infty} \tilde C_n I_\varepsilon(\tilde \eta)\mu_{Z,R}(d\eta)}{\int_{0}^{+\infty} \tilde C_n J_\varepsilon(\tilde \eta)\mu_{Z,R}(d\eta)}=\frac{ \int_0^1 d\tilde s \tilde s^2\big((1+\tilde s)(1-\tilde s)\big)^{\frac{n-3}{2}}}{ \int_0^1 d\tilde s \big((1+\tilde s)(1-\tilde s)\big)^{\frac{n-3}{2}}} \varepsilon^2 \frac{\int_0^{+\infty} \tilde \eta ^2(1+O((\varepsilon \tilde \eta)^2))\mu_{Z,R}(d\tilde \eta)}{\int_0^{+\infty} (1+O((\varepsilon \tilde \eta)^2))\mu_{Z,R}(d\tilde \eta)}\\
&=\frac{ \int_0^1 d\tilde s \tilde s^2\big(1-\tilde s^2\big)^{\frac{n-3}{2}}}{ \int_0^1 d\tilde s \big(1-\tilde s^2\big)^{\frac{n-3}{2}}} \varepsilon^2 \Big(\int_0^{+\infty} \tilde \eta ^2\mu_{Z,R}(d\tilde \eta)+O(\varepsilon^2)\Big).
\end{align*}
Setting for $m\in \N,\ I(m):=\int_0^\pi \sin(\theta)^m d\theta$ and changing variables setting $\tilde s=\cos(\theta) $, it is easily seen that:
$$\frac{ \int_0^1 d\tilde s \tilde s^2\big(1-\tilde s^2\big)^{\frac{n-3}{2}}}{ \int_0^1 d\tilde s \big(1-\tilde s^2\big)^{\frac{n-3}{2}}}=\frac{\int_0^\pi \cos^2(\theta)\sin(\theta)^{n-2}d\theta}{\int_0^\pi \sin(\theta)^{n-2}d\theta} =\frac{I(n-2)-I(n)}{I(n-2)}=1-\frac{I(n)}{I(n-2)}=1-\frac{n-1}{n}=\frac 1n,$$
where the last but one inequality follows from a direct integration by part. We have thus established that:
$$V_{Z_\varepsilon}=\frac{\varepsilon^2}n \Big(\int_0^{+\infty} \tilde \eta ^2\mu_{Z,R}(d\tilde \eta)+O(\varepsilon^2)\Big).$$
}
In particular,
 for $\varepsilon=\frac 1{\sqrt N} $ this precisely gives
\begin{equation*}
|V_{\oplus_{j=1}^N \frac 1{\sqrt N}\otimes Z^{j}}-t|\le \frac{C}{ N},
\end{equation*}
which is precisely \eqref{CONV_RATE_FOR_VARIANCE} and readily gives as well \eqref{VAR_LIM}. . Equation \eqref{VAR_LIM_LLN} is derived similarly using the corresponding specific scaling.
\end{proof}

\mysection{Proof of the main results} 
 \label{PROOFS}
 
 For a given $N\in \N$, let us consider the variables $\bar S_N$ and $S_N $ defined in \eqref{WALKS}. From
 \eqref{DEF_CONV} and Proposition \ref{PROP_CONV} we derive that they both possess a (radial) density, $f_{\bar S_N} $ and $f_{S_N} $ respectively, and:
 \begin{align}
 \label{PROP_CONV_WALK}
 f_{\bar S_N}&=f_{\frac 1{ N}\otimes Z^1}*\cdots*f_{\frac 1{ N}\otimes Z^N},\notag\\
 f_{S_N}&=f_{\frac 1{\sqrt N}\otimes Z^1}*\cdots*f_{\frac 1{\sqrt N}\otimes Z^N}.
 \end{align}

%
%
%
%
%
 We will denote by 
 \begin{equation}\label{GEO_RAD}
  \mu_{\varepsilon \otimes Z,R}(d\eta) = \Omega_{n-1}f_{\varepsilon \otimes Z}(\tanh(\frac \eta 2))(\sinh \eta)^{n-1} d\eta, \eta\ge 0,\ \varepsilon\in \{\frac{1}{\sqrt N},\frac 1N\}, 
  \end{equation}
   the radial density of $\varepsilon\otimes Z$ in geodesic polar coordinates. This corresponds for $\varepsilon=\frac 1{\sqrt N} $ to what in \cite{karp:tutu:shur:59} is denoted by $\hat\mu_{N,1} $. \\
 
 \subsection{Proof of the Theorem \ref{LLN}: Law of large numbers}
 \subsubsection{Proof of the main estimate}
 It suffices to establish the pointwise convergence of the characteristic function of $\bar S_N $  as defined in \eqref{CAR_FUNC} towards the constant $1$. From \eqref{PROP_CONV_WALK} and Proposition \ref{PROP_CONV}, we have for all $\lambda \in \R $:
  \begin{align}\label{FONC_CAR_SUM_LLN}
\Phi_{\bar S_{N}}(\lambda)=\frac{\hat f_{\bar S_{N}}(\lambda)}{\hat f_{\bar S_{N}}(0)}=\left( \frac{\hat f_{\frac 1N\otimes Z} (\lambda)}{\hat f_{\frac 1N\otimes Z}(0)}\right)^{N}=\exp\left( N \log\left(\frac{\hat f_{\frac 1N\otimes Z } (\lambda)}{\hat f_{\frac 1N\otimes Z}(0)}\right)\right)=\exp(N\log (\Phi_{\frac 1N\otimes Z}(\lambda))) .
 \end{align}
 From \eqref{0_ODD_MOMENTS}, we get
 \begin{align}\label{TAYLOR_FONC_CAR_LLN}
\Phi_{\frac 1N \otimes Z}(\lambda)=1-\lambda^2 \int_0^1 d\delta (1-\delta)\frac{\hat f_{\frac 1N \otimes Z}^{(2)}(\delta \lambda)}{\hat f_{\frac 1N\otimes Z}(0)}.
 \end{align}
We get similarly to the proof of Proposition \ref{VAR_OF_THE_WALKS}, \textcolor{black}{see e.g. \eqref{EQ_VAR_EPS} replacing therein $\lambda=0 $ by $\delta\lambda $}, that
$$\Big|\int_0^1 d\delta (1-\delta)\frac{\hat f_{\frac 1N \otimes Z}^{(2)}(\delta \lambda)}{\hat f_{\frac 1N\otimes Z}(0)}\Big|\le \textcolor{black}{\frac 12 V_{\frac 1N\otimes Z}\underset{\eqref{VAR_CTR}}{\le}} \frac{C}{N^2}, $$
which plugged into  \eqref{TAYLOR_FONC_CAR_LLN}, \eqref{FONC_CAR_SUM_LLN} gives $\Phi_{\bar S_{N}}(\lambda)\underset{N}{\rightarrow }1 $ which yields the stated convergence.

\subsubsection{Connection with the Sturm geodesic walks}
In the article \cite{stur:02}, Sturm proposed a very general construction of geodesic random walks in  non positive curvature (NPC) metric spaces and established a corresponding \textcolor{black}{\textit{weak}} and \textit{strong} law of large number. In particular, the Hyperbolic space $\mathbb H^n$ enters this setting.
 The construction proposed in \cite{stur:02} is the following. Let $N$ be a NPC metric space.
 For fixed $n\in \N^*$, let $\tilde S_{n-1}$ be given (representing the value of the normalized walk, with the scaling of the law of large numbers, at time $n-1$). Let $Z^n$ be the $n^{{\rm th}}$ innovation variable (defined on some probabilistic space $(\Omega,\mathcal A,\P) $ and $N$ valued). 
 
 Consider a a geodesic line $\big(x(t)\big)_{t\in [0,1]}$ s.t. $x(0)=\tilde S_{n-1}, x(1)=Z^n$. The walk is then updated setting 
$$\tilde S_n=x(\frac 1n) =:\Big(1-\frac 1n \Big) \tilde S_{n-1}+Z^n.
$$
We insist that the last equality is a definition and purely notational. Theorem 4.7 in \cite{stur:02} establishes that the walk asymptotically converges towards the \textit{mean} in law and probability (weak law of large numbers). The convergence is strong provided $Z^n$ is bounded. The \textit{mean} is intended here as the 
 \textit{barycenter} of the underlying law $\mu_Z $. If $Z\in L^2(\Omega,N)$ it corresponds to the unique minimizer of the mapping $z\mapsto \E[d^2(z,X)] $, where $d$ denotes here the distance between two points on $N$. We refer to Section 4 of \cite{stur:02} for additional details.

 Let us first describe in the current specific case of $\mathbb H^n$ how the previous construction can be related to the M\"obius addition and multiplication. We will actually prove that
\begin{equation}\label{GEO_WALK__WITH_MOBIUS}
 \tilde S_n=x(\frac 1n) =\tilde S_{n-1}\oplus \frac 1n\otimes (\ominus\tilde S_{n-1}\oplus Z^n),
 \end{equation}
 where for $a \in \B^n$, $\ominus a $ is the left inverse, i.e. $\ominus a\oplus a=0 $. Actually, it can be observed from \eqref{Mobius_ADD} that $\ominus a= -a$.
 To prove \eqref{GEO_WALK__WITH_MOBIUS} we will use the following result whose proof can be found in Section 4.1 of \cite{bara:unga:20}.
 \begin{proposition}[Geodesic lines and the M\"obius addition]\label{PROP_MOB_ADD}
 Let $a,b$ in $\B^n$. The geodesic line $\big(x(t) \big)_{t\in [0,1]}$  s.t. $x(0)=a$ and $x(1)=a\oplus b $ corresponds to the initial velocity
 $\dot x(0)=(1-\|a\|^2) \frac{\atanh(\|b\|)}{\|b\|}b $. Furthermore, for $t\in [0,1]$ it holds that $x(t)=a\oplus t\otimes b $.
 \end{proposition}
 
 To establish \eqref{GEO_WALK__WITH_MOBIUS}, which relates the geodesic walk $\tilde S_n $ in \cite{stur:02} to the M\"obius addition and multiplication, it then suffices to apply  Proposition \ref{PROP_MOB_ADD} with $a=\tilde S_{n-1}$ and $b=\ominus \tilde S_{n-1}\oplus Z^n$\footnote{to obtain the appropriate $b$ one has to solve $\tilde S_n\oplus b=Z^n $, which according to Proposition 3 in \cite{ferr:ren:11} admits the unique solution given by the indicated $b$.}, for which one indeed derives $x(1)=Z^n $ and to eventually take $t=\frac 1n$. The corresponding limit \textit{mean} in the previous sense is clearly 0 (center of the ball).
 
 Now, if the Möbius multiplication were distributive w.r.t. the Möbius addition, i.e.
 \begin{equation}\label{DIST}\tag{D}
 \frac 1n\otimes \Big( \ominus \tilde S_{n-1}\oplus Z^n\Big)=\frac 1n\otimes  \ominus \tilde S_{n-1}\oplus \frac 1n\otimes Z^n=-\frac 1n \otimes \tilde S_{n-1}\oplus \frac 1n \otimes Z^n,
 \end{equation}
 we would then obtain from \eqref{GEO_WALK__WITH_MOBIUS}, using as well that for $x\in \B^n, \lambda,\mu \in \R,\ \lambda \otimes x\oplus \mu\otimes x=(\lambda+\mu) \otimes x $ (see \cite{bara:unga:20}):
 $$\tilde S_n=(1-\frac 1n)\otimes \tilde S_{n-1}\oplus \frac 1n\otimes Z^n. $$
 A direct induction would then give $\tilde S_n=\oplus_{j=1}^n \frac 1n \otimes Z^j $ and the geodesic random walk would then correspond to the one we considered.
 
 However, in order to have the distributivity property \eqref{DIST} for the M\"obius multiplication, it is necessary and sufficient (within the setting of \textit{gyrogroups}) see e.g. \cite{kim:15}, that
 $${\rm gyr}[x,y]z:=-(x\oplus y)\oplus (x\oplus (y\oplus z))=z,$$
 i.e the gyration operator must be the identity. This property clearly fails in the current setting due to the non associativity of the M\"obius addition.
 
 Hence, the walks in \cite{stur:02} and the ones considered here are different objects (even though they converge to the same limit for the corresponding law of large numbers).

 \subsection{Proof of Theorem \ref{CLT}}
%
 \begin{proof}
 From \eqref{INV} it suffices to establish the convergence of the characteristic function  that we normalize as in \eqref{CAR_FUNC}. From \eqref{PROP_CONV_WALK} and Proposition \ref{PROP_CONV}, write for all $\lambda\in \R $:
 \begin{align}\label{FONC_CAR_SUM}
\Phi_{S_{N}}(\lambda)=\frac{\hat f_{S_{N}}(\lambda)}{\hat f_{S_{N}}(0)}=\left( \frac{\hat f_{\frac 1{\sqrt N}\otimes Z } (\lambda)}{\hat f_{\frac 1{\sqrt N}\otimes Z}(0)}\right)^{N}=\exp\left( N \log\left(\frac{\hat f_{\frac 1{\sqrt N}\otimes Z } (\lambda)}{\hat f_{\frac 1{\sqrt N}\otimes Z}(0)}\right)\right)=\exp(N\log (\Phi_{\frac 1{\sqrt N}\otimes Z}(\lambda))) .
 \end{align}
 From \eqref{0_ODD_MOMENTS}, \eqref{DEF_VAR} we derive that
 \begin{align}\label{TAYLOR_FONC_CAR}
\Phi_{\frac 1{\sqrt N}\otimes Z}(\lambda)=1-\frac{\lambda^2}{2} V_{\frac 1{\sqrt N}\otimes Z}+\frac{\lambda^4}{6}\int_0^1 d\delta (1-\delta)^3\frac{\hat f_{\frac 1{\sqrt N}\otimes Z}^{(4)}(\delta \lambda)}{\hat f_{\frac 1{\sqrt N}\otimes Z}(0)}=:1-\frac{\lambda^2}{2} V_{\frac 1{\sqrt N}\otimes Z}+\frac{\lambda^4}{6}R_N,
 \end{align}
where $\hat f_{\frac 1{\sqrt N}\otimes Z}^{(4)}$ stands for the fourth derivative of the Fourier-Helgason transform. The point is now, somehow as in a \textit{usual} limit theorem to control the remainder. We have:
\begin{align}
\left|\frac{\hat f_{\frac 1{\sqrt N}\otimes Z }^{(4)} (\delta\lambda)}{\hat f_{\frac 1{\sqrt N}\otimes Z}(0)}\right|\le& \frac{\int_0^{+\infty}\Big|\partial_\nu^4 \varphi_\nu(\tanh(\frac{\eta}{2}))_{\nu=\delta \lambda }\Big|\mu_{{\frac 1{\sqrt N}\otimes Z},R}(d\eta)}{\int_0^{+\infty} \varphi_0(\tanh(\frac{\eta}{2}))\mu_{{\frac 1{\sqrt N}\otimes Z},R}(d\eta)}\notag\\
=&\frac{\int_0^{\tau}\Big|\partial_\nu^4 \varphi_\nu(\tanh(\frac{\eta}{2}))_{\nu=\delta \lambda }\Big|\mu_{{\frac 1{\sqrt N}\otimes Z},R}(d\eta)}{\int_0^{+\infty} \varphi_0(\tanh(\frac{\eta}{2}))\mu_{{\frac 1{\sqrt N}\otimes Z},R}(d\eta)}+\frac{\int_\tau^{+\infty}\Big|\partial_\nu^4 \varphi_\nu(\tanh(\frac{\eta}{2}))_{\nu=\delta \lambda }\Big|\mu_{{\frac 1{\sqrt N}\otimes Z},R}(d\eta)}{\int_0^{+\infty} \varphi_0(\tanh(\frac{\eta}{2}))\mu_{{\frac 1{\sqrt N}\otimes Z},R}(d\eta)}=:I_1(\tau)+I_2(\tau),\label{DECOUP_INTEG_DU_RESTE}
\end{align}
for some $\tau>0 $ to be specified. Observing from \eqref{EIG_FUNC_GEO} that, for $\eta\in [0,\tau] $,
\begin{equation}\label{FROM_4_TO_2}
\Big|\partial_\nu^4 \varphi_\nu(\tanh(\frac{\eta}{2}))_{\nu=\delta \lambda }\Big|\le - \tau^2 \partial_\nu^2 \varphi_\nu(\tanh(\frac{\eta}{2}))_{\nu=0 },
\end{equation}
write:
\begin{align*}
I_1(\tau)= V_{\frac 1{\sqrt N}\otimes Z}\frac{\int_0^{\tau}\Big|\partial_\nu^4 \varphi_\nu(\tanh(\frac{\eta}{2}))_{\nu=\delta \lambda }\Big|\mu_{{\frac 1{\sqrt N}\otimes Z},R}(d\eta)}{-\int_0^{\infty}\partial_\nu^2 \varphi_\nu(\tanh(\frac{\eta}{2}))_{\nu=0 }\mu_{{\frac 1{\sqrt N}\otimes Z},R}(d\eta)} \le V_{\frac 1{\sqrt N}\otimes Z} \tau^2. 
\end{align*}
Similarly to \eqref{FROM_4_TO_2} we derive that for $\eta\in [\tau,+\infty) $:
$$\Big|\partial_\nu^4 \varphi_\nu(\tanh(\frac{\eta}{2}))_{\nu=\delta \lambda }\Big|\le - \eta^2 \partial_\nu^2 \varphi_\nu(\tanh(\frac{\eta}{2}))_{\nu=0 }.
 $$

From the definition of $\mu_{\frac 1{\sqrt N}\otimes Z,R} $ introduced in \eqref{GEO_RAD} and Proposition \ref{SCALED_KTS} we derive
a tail control of the Fourier-Helgason transform. Namely, for any $\tau>0 $ it holds  that 
\begin{equation}\label{CTR_TAILS_SCALE}
\int_{\tau}^{\infty} \eta^2\partial_\nu^2 \varphi_\nu(\tanh(\frac \eta 2))|_{\nu=0} \mu_{\frac 1{\sqrt N}\otimes Z,R}(d\eta)\underset{N}{\longrightarrow} 0.
\end{equation}

We then get from \eqref{CTR_TAILS_SCALE}:
$$I_2(\tau)\le V_{\frac 1{\sqrt N}\otimes Z}  \frac{\int_{\tau}^{\infty}\eta^2\partial_\nu^2 \varphi_\nu(\tanh(\frac{\eta}{2}))_{\nu=0}\mu_{{\frac 1{\sqrt N}\otimes Z},R}(d\eta)}{\int_0^{\infty}\partial_\nu^2 \varphi_\nu(\tanh(\frac{\eta}{2}))_{\nu=0 }\mu_{{\frac 1{\sqrt N}\otimes Z},R}(d\eta)} =:V_{\frac 1{\sqrt N}\otimes Z}  \mathcal R_N(\tau), \ \mathcal R_N(\tau)\underset{N}{\longrightarrow} 0, \tau>0.$$
 
 Plugging the above controls for $I_1(\tau),I_2(\tau) $ into \eqref{DECOUP_INTEG_DU_RESTE} we derive that the remainder term $R_N$ in \eqref{TAYLOR_FONC_CAR} enjoys the upper-bound:
 $$ R_N\le \frac{ V_{\frac 1{\sqrt N}\otimes Z}}4(\tau^2+\mathcal R_N(\tau)).$$
 
Expanding then the logarithm in \eqref{FONC_CAR_SUM}, we thus get:
 \begin{align*}
 \Phi_{S_{N}}(\lambda)&=\exp\left( -\frac{\lambda^2}2 V_{\frac 1{\sqrt N}\otimes Z}N +\frac{\lambda ^4}6N R_N+o(1)\right)\\
&= \exp\left( -\frac{\lambda^2}2 t +\frac{\lambda ^4}6 t \frac{R_N}{V_{\frac 1{\sqrt N}\otimes Z}}+o(1)\right).
 \end{align*}
Taking $\tau$ small enough and then $N$ large enough we deduce that:
$$\Phi_{S_{N}}(\lambda)\underset{N}{\longrightarrow} \exp(-\frac{\lambda^2}2t) =\hat \Psi(t,\lambda),$$
 using \eqref{TF_HK_2}, i.e. we recognize that the limit is the characteristic function of the second kind of the \textit{normal} law on $\B^n$. The statement then follows from the inversion formula \eqref{INV} and domination arguments.

 \end{proof}
 
 \subsection{Proof of the Local limit Theorem \ref{LLT}}
 
 The point is now to go further than the convergence in law for the random walk established in Theorem \ref{CLT} and to establish under the previous assumptions  a (global) pointwise control for the difference of the densities. 
 From \eqref{INV} the difference writes:
 \begin{align}
 &f_{S_{N}}(\tanh(\frac \eta2))-\Psi(t,\tanh(\frac \eta 2))\notag\\
 =&C_n\int_0^{+\infty} [\hat f_{S_{N}}(\lambda)-\hat \psi(t,\lambda)]\varphi_\lambda(\tanh(\frac \eta 2))|\c(\lambda)|^{-2}d\lambda\notag\\
 =&C_n\int_0^{+\infty} [\prod_{j=1}^{N}\hat f_{\frac{1}{\sqrt N}\otimes Z^{j}}(\lambda)-\exp(-\frac{(\rho^2+\lambda^2)t}{2})]\varphi_\lambda(\tanh(\frac \eta 2))|\c(\lambda)|^{-2}d\lambda\notag\\
 =&C_n\int_0^{+\infty} (\I_{\lambda\le D_N}+\I_{\lambda>D_N})[(\hat f_{\frac 1{\sqrt N}\otimes Z}(\lambda))^N-\exp(-\frac{(\rho^2+\lambda^2)t}{2})]\varphi_\lambda(\tanh(\frac \eta 2))|\c(\lambda)|^{-2}d\lambda
 =:\textcolor{black}{(\mathcal B_N+\mathcal T_N)(\eta)},\notag\\
 \label{THE_DECOUP_LLT}
\end{align}
where $D_N$ is a cutting level to be specified which will allow to balance the contribution for the terms \textcolor{black}{$\mathcal B_N(\eta) $  and $\mathcal T_N(\eta) $} corresponding respectively to the bulk and tails of the Fourier-Helgasson integral.
\begin{trivlist}
\item[-] \textbf{Contribution of the bulk}.
We first concentrate on the difference of the two Fourier transforms in the bulk, i.e. for $\lambda\le D_N $. Namely, write:
\begin{align}
&\Big| (\hat f_{\frac 1{\sqrt N}\otimes Z}(\lambda))^N-\exp(-\frac{(\rho^2+\lambda^2)t}{2})\Big|\notag\\
=&\exp(-\frac{(\rho^2+\lambda^2)t}{2})\times \Big|\exp\left(N \ln\Big(\hat f_{\frac 1{\sqrt N}\otimes Z}(\lambda)\Big)+\frac{(\lambda^2+\rho^2)t}2 \right) -1\Big|\notag\\
\le &\exp(-\frac{(\rho^2+\lambda^2)t}{2})\Big|N \ln\Big(\hat f_{\frac 1{\sqrt N}\otimes Z}(\lambda)\Big)+\frac{(\lambda^2+\rho^2)t}2 \Big|\exp\Big(\Big|N \ln\Big(\hat f_{\frac 1{\sqrt N}\otimes Z}(\lambda)\Big)+\frac{(\lambda^2+\rho^2)t}2 \Big|\Big),\label{PREAL_BULK}
\end{align}
using that for all $x\in \mathbb C$, $ |e^x-1|\le |x|\exp(|x|)$ for the last inequality.

Write now:
\begin{align}
&\Big|N \ln\Big(\hat f_{\frac 1{\sqrt N}\otimes Z}(\lambda)\Big)+\frac{(\lambda^2+\rho^2)t}2\Big|=\Big|N\ln\Big(\hat f_{\frac 1{\sqrt N}\otimes Z}(0)\Big( 1-(1-\frac{\hat f_{\frac1{\sqrt N} \otimes Z}(\lambda)}{\hat f_{\frac1{\sqrt N}\otimes Z}(0)})\Big)\Big)+\frac{(\lambda^2+\rho^2)t}2\Big|\notag\\
\le & |N\ln\Big(\hat f_{\frac 1{\sqrt N}\otimes Z}(0)\Big)+\frac{\rho^2}2t|+\frac{\lambda^2}2 | t-N V_{\frac 1{\sqrt N}\otimes Z}|+N |(\ln\Big( 1-(1-\frac{\hat f_{\frac1{\sqrt N}\otimes Z}(\lambda)}{\hat f_{\frac 1{\sqrt N}\otimes Z}(0)})\Big)+\frac{\lambda^2}2V_{\frac 1{\sqrt N}\otimes Z}|\notag\\
\le &|N\ln\Big(\hat f_{\frac{1}{\sqrt N}\otimes Z}(0)\Big)+\frac{\rho^2}2t|+\frac{\lambda^2}2 | t-V_{S_{N}}|+ N\Big(\sum_{r=2}^{\infty}r^{-1}|1-\frac{\hat f_{\frac1{\sqrt N}\otimes Z}(\lambda)}{\hat f_{\frac1{\sqrt N}\otimes Z}(0)}|^r +|-1+\frac{\hat f_{\frac1{\sqrt N}\otimes Z}(\lambda)}{\hat f_{\frac1{\sqrt N}\otimes Z}(0)}+\frac{\lambda^2}2V_{\frac1{\sqrt N}\otimes Z}\Big|\Big)\notag\\
=:&\mathcal R_{N,1}+\mathcal R_{2,N}+\mathcal R_{3,N}.\label{LE_CTR_POUR_LE_LOG}
\end{align}
We can write similarly to \eqref{TAYLOR_FONC_CAR}, \textcolor{black}{expanding again up to order 4},  that 
 \begin{align}\label{TAYLOR_FONC_CAR_BIS}
\frac{\hat f_{\frac1{\sqrt N}\otimes Z}(\lambda)}{\hat f_{\frac1{\sqrt N}\otimes Z}(0)}=1-\frac{\lambda^2}{2} V_{\frac 1{\sqrt N}\otimes Z}+\frac{\lambda^4}{6}\int_0^1 d\delta (1-\delta)^3\frac{\hat f_{\frac 1{\sqrt N}\otimes Z}^{(4)}(\delta \lambda)}{\hat f_{\frac 1{\sqrt N}\otimes Z}(0)}=:1-\frac{\lambda^2}{2} V_{\frac 1{\sqrt N}\otimes Z}+\frac{\lambda^4}{6} R_{N}.
 \end{align}
 From \eqref{EIG_FUNC_GEO} it follows that $\max_{\delta\in [0,1]}|\varphi_{\delta\lambda}^{(4)}(\tanh (\frac \eta 2))|\le \frac {\eta^2}2|\varphi_0^{(2)}(\tanh(\frac \eta 2))| $ and therefore
 \begin{equation}
\label{CTR_RESTE_LLT_1}
| R_{N}|\le \E[|Z_{\frac1{\sqrt N},R}|^2\frac{\varphi_0^{(2)}\big(\tanh(\frac{Z_{\frac 1{\sqrt N},R}}2)\big)}{\hat f_{\frac 1{\sqrt N}\otimes Z}(0)}],
 \end{equation}
 where $Z_{\frac1{\sqrt N},R} $ is a random variable with law $\mu_{\frac 1{\sqrt N}\otimes Z,R} $ introduced in \eqref{GEO_RAD} corresponding to the geodesic polar coordinates of $\frac 1{\sqrt N}\otimes Z$. 
 It follows that:
 \begin{align}
&\E[|Z_{\frac 1{\sqrt N},R}|^2\frac{|\varphi_0^{(2)}\big(\tanh(\frac{Z_{\frac 1{\sqrt N},R}}2)\big)|}{\hat f_{\frac 1{\sqrt N}\otimes Z}(0)}]\notag\\
=&\Omega_{n-1}\int_{0}^{+\infty} \eta^2 \frac{|\varphi_0^{(2)}(\tanh(\frac \eta 2))|}{\hat f_{ \frac 1{\sqrt N}\otimes Z}(0)} f_{\frac 1{\sqrt N}\otimes Z}(\tanh (\frac \eta2))\sinh(\eta)^{n-1} d\eta\notag\\
\underset{\eqref{SCALED_DENS_PROP}}{=}&\Omega_{n-1}\int_{0}^{+\infty} \eta^2 \frac{|\varphi_0^{(2)}(\tanh(\frac \eta 2))|}{\hat f_{\frac 1{\sqrt N}\otimes Z}(0)} \sqrt Nf_{ Z}(\tanh (\sqrt N\frac \eta2))\Big(\frac{\sinh(\sqrt N \eta)}{\sinh(\eta)} \Big)^{n-1}\sinh(\eta)^{n-1} d\eta\notag\\
=&\Omega_{n-1}\frac1{\hat f_{\frac 1{\sqrt N}\otimes Z}(0)} \int_0^{+\infty} \frac{\tilde \eta^2}{N}|\varphi_0^{(2)}(\tanh(\frac{\tilde \eta}{2N^{\frac 12}}))|f_{ Z}(\tanh(\frac{\tilde \eta}2)) \sinh(\tilde \eta)^{n-1} d\tilde \eta\notag\\
\le &C\Omega_{n-1}\frac1{\hat f_{\frac 1{\sqrt N}\otimes Z}(0)} \int_0^{+\infty} \frac{\tilde \eta^4}{N^{2}}|\varphi_0(\tanh(\frac{\tilde \eta}{2N^{\frac 12}}))|f_{ Z}(\tanh(\frac{\tilde \eta}2)) \sinh(\tilde \eta)^{n-1} d\tilde \eta\le \frac{C}{N^{2}}.\label{THE_B_N}
 \end{align}
It therefore holds that there exists $C$ s.t. 
\begin{equation}
\label{CTR_RESTE_CUBE}
| R_{N}|\le   \frac{C}{N^{2}}.
\end{equation}
Set now,
\begin{align*}
A_N:=(V_{\frac 1{\sqrt N}\otimes Z})^{-\frac 14},\ B_N:=\Big(N\E\Big[|Z_{\frac 1{\sqrt N},R}|^2\frac{|\varphi_0^{(2)}\big(\tanh(\frac{Z_{\frac 1{\sqrt N},R}}2)\big)|}{\hat f_{\frac1{\sqrt N}\otimes Z}(0)}\Big]\Big)^{-\frac 14},\ C_N:=|t-V_{S_N}|^{-\frac 12}.
\end{align*}
In the previous setting, from \eqref{VAR_CTR}, \eqref{THE_B_N}, \eqref{CONV_RATE_FOR_VARIANCE},
\begin{equation}\label{CTR_RATES}
A_N^{-2}+ B_N^{- 2}+C_N^{- 1}\le CN^{-\frac 12}.
\end{equation}
With these notations, it is clear from the definitions in \eqref{LE_CTR_POUR_LE_LOG} \textcolor{black}{and \eqref{CTR_RATES}} that
\begin{equation}\label{CTR_R2N}
 \mathcal R_{2,N}\le \frac{\lambda^2}2C_N^{-2}\le C\frac{\lambda^2}N.
 \end{equation}
On the other hand,
\begin{align*}
\mathcal R_{3,N} \le N \Bigg(|1-\frac{\hat f_{\frac1{\sqrt N}\otimes Z}(\lambda)}{\hat f_{\frac1{\sqrt N}\otimes Z}(0)}|\sum_{r=2} r^{-1}|1-\frac{\hat f_{\frac1{\sqrt N}\otimes Z}(\lambda)}{\hat f_{\frac1{\sqrt N}\otimes Z}(0)}|^{r-1}+\frac{\lambda^4}{6}\E\left[Z_{\frac 1{\sqrt N},R}^2\Big|\frac{\varphi_0^{(2)}\big(\tanh (\frac{Z_{\frac 1{\sqrt N},R}} 2)\big)}{\hat f_{\frac1{\sqrt N}\otimes Z}(0)} \Big| \right]\Bigg) =:\mathcal R_{3,N}^{(1)}+\mathcal R_{3,N}^{(2)}.\label{T_R_III}
 \end{align*}
 Assume now that $|\lambda|\le D_N=
 \min_{} (\frac{A_N}{\sqrt 3},B_N,\frac{C_N}{\sqrt 3})
 $.
It readily follows from \eqref{THE_B_N} that:
 \begin{align*}
\mathcal R_{3,N}^{(2)}\le \frac{\lambda^4}6B_{N}^{-4}\le C\frac{\lambda^4} N.
 \end{align*}

We also have from \eqref{TAYLOR_FONC_CAR_BIS}, \eqref{CTR_RESTE_LLT_1} that:
 \begin{align*}
|\frac{\hat f_{\frac1{\sqrt N}\otimes Z}(\lambda)}{\hat f_{\frac1{\sqrt N}\otimes Z}(0)}-1|=\frac{\lambda^2}{2} V_{\frac 1{\sqrt N}\otimes Z}+\frac{\lambda^4}{6} R_{N}\le \frac 16+\frac 1{6} =\frac 13.
 \end{align*}
\textcolor{black}{We point out that this is a very coarse estimate which is only given to ensure the convergence of the series appearing for the term $\mathcal R_{3,N}^{(1)}$}.
 
 Then,
 \begin{align*}
\mathcal R_{3,N}^{(1)}&\le N|1-\frac{\hat f_{\frac1{\sqrt N}\otimes Z}(\lambda)}{\hat f_{\frac1{\sqrt N}\otimes Z}(0)}|^{2}\sum_{r=2} r^{-1}|1-\frac{\hat f_{\frac1{\sqrt N}\otimes Z}(\lambda)}{\hat f_{\frac1{\sqrt N}\otimes Z}(0)}|^{r-2}\le  CN (\lambda^4 V_{\frac 1{\sqrt N}\otimes Z}^{2}+\lambda^{8}  R_{N}^{2})\sum_{\tilde r\in \N} (2+\tilde r)^{-1} (\frac 13)^{\tilde r} \\
&\le C\Big( \lambda^4 NA_N^{- 8}+\frac{\lambda^{4}}{N^{2}}\times \lambda^4B_N^{-4}\Big)
\textcolor{black}{\le \lambda^4N^{-1}},
 \end{align*}
 exploiting \textcolor{black}{\eqref{CTR_RATES}, which gives that $N A_N^{-4}\le C$ and $A_N^{-4}\le C/N $,  as well as $\lambda^4B_N^{-4}\le 1 $} for the last inequality.  \textcolor{black}{We carefully mention that it is the  term involving the variance which is responsible for the above rate}.
 
 We have thus established that 
 \begin{equation}\label{CTR_R3N}
 \mathcal R_{3,N}
  \textcolor{black}{\le C\lambda^4N^{-1}}.
 \end{equation}
 It remains to handle the contribution $\textcolor{black}{\mathcal R_{1,N}} $ in \eqref{LE_CTR_POUR_LE_LOG}. Write:
 \begin{align}
\mathcal R_{1,N}&= |N\ln\Big(\hat f_{\frac1{\sqrt N}\otimes Z}(0)\Big)+\frac{\rho^2}2t|=|N  \ln\big(1-(1-\hat f_{\frac1{\sqrt N}\otimes Z}(0))\big)+\frac{\rho^2}2t|\nonumber\\
&=\big|-N (1-\hat f_{\frac1{\sqrt N}\otimes Z}(0))+\frac{\rho^2}2t- N(1-\hat f_{\frac1{\sqrt N}\otimes Z}(0))\sum_{r=2}^\infty r^{-1}(1-\hat f_{\frac1{\sqrt N}\otimes Z}(0))^{r-1}\big|.\label{MAIN_0}
 \end{align} 
Write now from \eqref{FUNC_MEAN_SCALED}:
\begin{align*}
\textcolor{black}{\hat f_{\frac1{\sqrt N}\otimes Z}(0)-1}=\Omega_{n-1}\int_0^{+\infty} f_Z(\tanh{\frac{\tilde \eta}{2}}) \big(\textcolor{black}{\varphi_0(\tanh(\frac{\tilde \eta}{2\sqrt N}))-1}\big)\sinh(\tilde \eta)^{n-1} d\tilde \eta.
\end{align*}
Setting now, $g(\zeta)=\varphi_0(\tanh( \frac{ \zeta}{2})) $, we perform a fourth order Taylor expansion which yields:
\begin{align*}
\textcolor{black}{\hat f_{\frac1{\sqrt N}\otimes Z}(0)-1}=&\Omega_{n-1}\int_0^{+\infty} f_Z(\tanh{\frac{\tilde \eta}{2}}) \Big(g^{(1)}(0)\frac{\tilde \eta}{ N^{\frac 12}}+\frac 12g^{(2)}(0)\frac{\tilde \eta^2}{N}+\frac 16g^{(3)}(0)\frac{\tilde \eta^3}{N^{\frac 32}} \\
&+\frac 16\int_0^1 g^{(4)}(\gamma \tilde \eta) (1-\gamma)^3d\gamma \frac{\tilde \eta^4}{N^2}\Big)\sinh(\tilde \eta)^{n-1} d\tilde \eta,
\end{align*}
where $g^{(i)} $ denotes the $ i^{{\rm th}}$ derivative of $g$. It now remains to compute the derivatives of $g$. We will actually prove that, the derivatives of odd index are zero at zero and we aim at retrieving for the second order derivative a contribution, which once summed in \eqref{MAIN_0} will precisely exactly give the contribution $-\frac {\rho^2}2t $.

To compute the derivative of $g$ let us recall from \eqref{EIG_FUNC_GEO} that:
 \begin{align*}
g(\eta)= \varphi_0 (\tanh(\frac \eta 2))=
 \tilde C_n\sinh(\eta)^{2-n}\int_0^\eta ds(\cosh(\eta)-\cosh(s))^{\frac {n-3}2},\ 
\tilde C_n =
\frac{2^{\frac {n-1}2}\Gamma\left(\frac n2\right)}{\sqrt \pi \Gamma\left(\frac{n-1}2\right)},
 \end{align*}
which also equivalently rewrites, see e.g. \cite{anke:17},
\begin{align*}
g(\eta)=\varphi_0 (\tanh(\frac \eta 2))=k_n\int_0^\pi (\cosh (\eta)-\sinh(\eta)\cos(\theta))^{-\rho}\sin(\theta)^{2\rho-1}d\theta,\ k_n:=\frac{\Gamma\left(\frac n2\right)}{\sqrt \pi \Gamma\left(\frac{n-1}2\right)}=\frac{\tilde C_n}{2^{\frac{n-1}2}}.
\end{align*}
From the above expression we indeed derive:
\begin{align*}
g^{(1)}(0)=&- k_n\rho\int_0^\pi  (\cosh (\eta)-\sinh(\eta)\cos(\theta))^{-(\rho+1)}(\sinh(\eta)-\cosh(\eta)\cos(\theta)) \sin(\theta)^{2\rho-1}d\theta|_{\eta=0}\\
=&k_n\rho\int_0^\pi \cos(\theta) \sin(\theta)^{n-2}\textcolor{black}{d\theta}=0,\\
g^{(2)}(0)=&k_n\rho(\rho+1)\int_0^\pi  (\cosh (\eta)-\sinh(\eta)\cos(\theta))^{-(\rho+2)}(\sinh(\eta)-\cosh(\eta)\cos(\theta))^2 \sin(\theta)^{2\rho-1}d\theta|_{\eta=0}\\
&-k_n\rho\int_0^\pi (\cosh (\eta)-\sinh(\eta)\cos(\theta))^{-(\rho+1)}(\cosh(\eta)-\sinh(\eta)\cos(\theta)) \sin(\theta)^{2\rho-1}d\theta|_{\eta=0}\\
=&k_n\rho(\rho+1)\int_0^\pi \cos(\theta)^2\sin(\theta)^{n-2}d\theta-k_n\rho\int_0^\pi \sin(\theta)^{n-2}d\theta\\
=&-k_n\rho(\rho+1)I(n)+k_n\rho^2I(n-2),\ \forall m\in \N,\ I(m):=\int_0^\pi \sin(\theta)^{m}d\theta,\\
g^{(3)}(0)=&0.
\end{align*}
Hence,
\begin{align*}
\textcolor{black}{\hat f_{\frac1{\sqrt N}\otimes Z}(0)-1}=\frac{k_n(-\rho(\rho+1)I(n)+\rho^2I(n-2))}{2N}\int_0^{+\infty} \tilde \eta^2\mu_{Z,R}(d\eta)+O(\frac 1{N^2}).
\end{align*}
Having in mind the explicit expression of $t$ in \eqref{VAR_LIM}, to investigate the difference of the first l.h.s term in \eqref{MAIN_0} it remains to specify properly the constants involved (in both the definition of $t$ and the above expression).

From \cite{grad:ryzh:07}, formula 3.631 (8), p. 386, it holds that for $\nu\in \C, \ \Re(\nu)>0 $,
$$\int_0^\pi \sin(\theta)^{\nu-1}d\theta=\frac{\pi}{2^{\nu-1}\nu B(\frac{\nu+1}{2},\frac{\nu+1}{2})},$$
where $B(\cdot,\cdot) $ denotes the $\beta $-function. Therefore taking respectively $\nu=n+1 $ and $\nu=n-1 $, recalling as well that $B(x,y)=\frac{\Gamma(x)\Gamma(y)}{\Gamma(x+y)} $, one derives
\begin{align*}
I(n)=\frac{\pi \Gamma(n+2)}{2^n (n+1)\Gamma(\frac{n+2}2)^2}=\frac{\pi n!}{2^n\Gamma(\frac{n+2}2)^2},\ I(n-2)=\frac{\pi (n-2)!}{2^{n-2}\Gamma(\frac{n}2)^2}.
\end{align*}
Recalling from \eqref{TF_HK_1} that $\rho=\frac{n-1}2 $, this in turn gives:
\begin{align*}
-k_n\rho(\rho+1)I(n)&=-\frac{\Gamma(\frac n2)}{\sqrt \pi \Gamma(\frac {n-1}2)}\frac{n-1}2\frac{n+1}2\frac{\pi n!}{2^n\left(\frac{n}{2}\right)^2\Gamma(\frac{n}2)^2}\\
&=-\frac{n^2-1}{2^n n^2 \Gamma(\frac{n-1}{2}) \Gamma(\frac{n}{2})}\sqrt \pi n!=-\frac{n^2-1}{4\Gamma(n-1)}\frac{n!}{n^2}=-\frac{(n^2-1)}{4}\frac{n-1}{n}=-\frac{\rho^2(n+1)}{n}\\
\end{align*}
using the Legendre duplication formula for the third equality. We also similarly derive,
\begin{align*}
k_n\rho^2I(n-2)&=\frac{\Gamma(\frac n2)}{\sqrt \pi \Gamma(\frac {n-1}2)}\Big(\frac{n-1}2 \Big)^2\frac{\pi (n-2)!}{2^{n-2}\Gamma(\frac{n}2)^2}=\frac{(n-2)!}{\Gamma(n-1)}\Big(\frac{n-1}2 \Big)^2=\rho^2.
\end{align*}
We have thus established:
\begin{align*}
1-\hat f_{\frac1{\sqrt N}\otimes Z}(0)=\textcolor{black}{\frac{\rho^2}{n}}\frac{1}{2N}\int_0^{+\infty} \tilde \eta^2\mu_{Z,R}(d\eta)+O(\frac 1{N^2}),
\end{align*}
which eventually gives in \eqref{MAIN_0}
\begin{align*}
\mathcal R_{1,N}=-\frac{\rho^2}{2}\frac{1}{n}\int_0^{+\infty} \tilde \eta^2\mu_{Z,R}(d\eta)+O(\frac 1{N})+\frac{\rho^2}2t.
\end{align*}
This eventually gives \textcolor{black}{from the very definition of $t$ (see e.g. \eqref{VAR_LIM})}:
\begin{align} \label{CTR_R1N}
\mathcal R_{1,N}=O(\frac 1N).
\end{align}

Therefore, plugging \eqref{CTR_R1N},  \eqref{CTR_R3N}, \eqref{CTR_R2N} into \eqref{LE_CTR_POUR_LE_LOG} and \eqref{PREAL_BULK} we thus derive, choosing $D_N=N^{\frac 14}$ so that in particular $\lambda^4/N\le 1 $:
\begin{align}
|\textcolor{black}{\mathcal B_N(\eta)}|\le& \textcolor{black}{C}\int_{0<\lambda \le D_N }\exp(-\frac{(\rho^2+\lambda^2)t}{2}) \frac 1N\big(1+\lambda^4 \big)\exp\Big( \frac 1N\big(1+\lambda^4 \big)\Big)|\varphi_\lambda (\tanh(\frac \eta 2))||\c(\lambda)|^{-2}d\lambda \notag\\
\le &\frac C{t^2N} \int_{0<\lambda \le D_N }\exp(-\frac{(\rho^2+\lambda^2)t}{4}) |\varphi_\lambda (\tanh(\frac \eta 2))| |\c(\lambda)|^{-2}d\lambda\notag\\
\le &\frac C{t^2N} \int_{\R^+ }\exp(-\frac{(\rho^2+\lambda^2)t}{4}) |\varphi_\lambda (\tanh(\frac \eta 2))| |\c(\lambda)|^{-2}d\lambda.\label{CTR_BULK_PREAL}
\end{align}

From \eqref{EIG_FUNC_GEO} and the definition of the Harish-Chandra function in \eqref{HC}, see e.g. Remark 2.7 in \cite{anke:17} for its asymptotic behavior, we get 
 \begin{equation}\label{ASYMP_HARISH_CHANDRA_ET_PHI}
\ |\varphi_\lambda (\tanh(\frac \eta 2))|\le \varphi_0(\tanh(\frac \eta 2))\le C,\ |\c(\lambda)|^{-2}\le C(\lambda^2\I_{\lambda\le C}+\lambda^{\textcolor{black}{n-1}}\I_{\lambda>C}),\ 
 \end{equation}
from which we eventually derive:
\begin{align}
|\textcolor{black}{\mathcal B_N(\eta)}|\le& \frac C{t^{2}N}\textcolor{black}{\Big(\frac 1{t^{\frac 12}}\wedge \frac{1}{t^{\frac n2}} \Big)}. \label{CTR_BULK}
\end{align}
\item[-] \textbf{Contribution of the tails, general case.}
 For $\lambda >N^{1/4} $, we write:
 \begin{align}
\textcolor{black}{\mathcal T_N(\eta)}\le& \Big|\int_{N^{1/4}}^{+\infty} \Big(\prod_{j=1}^{N}\hat f_{\frac1{\sqrt N}\otimes Z^j}(\lambda)\Big) \varphi_\lambda (\tanh(\frac \eta 2))|\c(\lambda)|^{-2}d\lambda\Big|+\Big|\int_{N^{1/4}}^{+\infty}  \exp(-\frac{(\rho^2+\lambda^2)t}{2}) \varphi_\lambda (\tanh(\frac \eta 2))|\c(\lambda)|^{-2} d\lambda\Big|\notag\\
=:&\textcolor{black}{(\mathcal T_N^1+\mathcal T_N^2)(\eta)}.\label{DECOUP_TAILS_GEN}
\end{align}

Let us first consider the term $\textcolor{black}{\mathcal T_N^2}$ which can be handled globally. We get again from \eqref{ASYMP_HARISH_CHANDRA_ET_PHI} \textcolor{black}{(similarly to \eqref{CTR_BULK})}:
\begin{align}
\textcolor{black}{\mathcal T_N^2(\eta)}\le &\exp(-\textcolor{black}{\frac{N^{\frac 12}t}{4}})\int_{0}^{+\infty}  \exp(-\frac{(\rho^2+\lambda^2)t}{4}) | \varphi_\lambda (\tanh(\frac \eta 2))| |\c(\lambda)|^{-2} d\lambda\notag\\
\le & \frac{C}{t^2N }\textcolor{black}{\Big(\frac 1{t^{\frac 12}}\wedge \frac{1}{t^{\frac n2}} \Big)},\label{CTR_TAILS_HK}
 \end{align}
 which gives an upper bound similar to the one obtained for the bulk in \eqref{CTR_BULK}.

Let us now turn to $\textcolor{black}{\mathcal T_N^1(\eta),}$ and split as follows:
 \begin{align}
\textcolor{black}{\mathcal T_N^1(\eta)}\le&  \Big|\int_{N^{\frac 14}}^{c_0N^{\frac 12}}\Big(\hat f_{\frac1{\sqrt N}\otimes Z}(\lambda)\Big)^N\varphi_\lambda (\tanh(\frac \eta 2))|\c(\lambda)|^{-2}d\lambda\Big|\notag\\
&+\Big|\int_{c_0N^{\frac 12}}^{\infty}\Big(\hat f_{\frac1{\sqrt N}\otimes Z}(\lambda)\Big)^N \varphi_\lambda (\tanh(\frac \eta 2))|\c(\lambda)|^{-2}d\lambda\Big|=:\textcolor{black}{(\mathcal T_N^{11}+\mathcal T_N^{12})(\eta)},\label{SPLIT_TAILS}
\end{align}
for some small enough constant $c_0$ to be specified. 

Let us now recall that from Lemma 2.1 of \cite{liu:peng:04}, for $x\in \B^n$ and $\lambda\in \C $
\begin{align*}
\int_{\S^{n-1}}\frac{\sigma(d\zeta)}{|x-\zeta|^{2\lambda}}={}_2F_1(\lambda,\lambda-\frac n2+1,\frac n2,\|x\|^2),
\end{align*}
where ${}_2F_1$ stands for the hypergeometric function and $\textcolor{black}{\sigma=\Lambda/\Omega_{n-1}}$ for the \textcolor{black}{normalized} surface measure on $\S^{n-1} $ \textcolor{black}{(i.e. $\sigma(\mathbb S^{n-1})=1$)}. 
We now recall, \textcolor{black}{see e.g. formula 15.1.1. in \cite{grad:ryzh:07}}, \textcolor{black}{that for $a,b,c,z$, where $c $ is not a negative integer and $ |z|<1$,
\begin{equation}\label{SERIES}
{}_2F_{1}(a,b,c,z)= \sum_{q=0}^{+\infty} \frac{(a)_q(b)_q}{(c)_q}\frac{z^q}{q!},
\end{equation}
denoting for $d\in \{a,b,c\}$,
$$(d_q)=\begin{cases} 
1,\ q=0,\\
\prod_{j=0}^{q-1}(d+j),\ q\neq 0.
\end{cases}
$$
The series gives an analytic function if $|z|<1 $ and can be continued on the complex plane (cut on $[1,+\infty))$.
}

From \eqref{DEF_Z}, \eqref{EIG_SPHERE_HARMO}, we thus get:
\begin{align*}
\varphi_\lambda(x)=(1-\|x\|^2)^{\rho+i\frac\lambda2}\int_{\S^{n-1}}\frac{\sigma(d\zeta)}{\|x-\zeta\|^{2(\rho+i\frac \lambda 2)}}=(1-\|x\|^2)^{\rho+i\frac\lambda2}{}_2F_{1}(\rho+i\frac\lambda 2,\frac 12+i\frac \lambda 2,\frac n2,\|x\|^2).
\end{align*}
Recalling the Pfaffian identity (see e.g. (2.15) in \cite{liu:peng:09})
$${}_2F_{1}(a,b,c,z)=(1-z)^{-a}{}_2F_{1}(a,c-b,c,\frac z{z-1}),$$
we get for $z=\|x\|^2=\tanh(\frac\eta 2)^2 $, $a=\rho+i\frac \lambda 2,b=\frac 12+i\frac \lambda 2,c=\frac n2$,

$$\varphi_\lambda(\tanh(\frac \eta 2))={}_2F_1(\rho+i\frac \lambda 2, \rho-i\frac \lambda 2,\frac n2, -\sinh ^2(\frac \eta 2)).$$
\textcolor{black}{
We therefore get from \eqref{SERIES}: 
\begin{align}
\varphi_\lambda(\tanh(\frac \eta 2))=&1+\sum_{q=1}^{\infty}\frac{\prod_{j=0}^{q-1}\big((\rho+i\frac \lambda 2+j)(\rho-i\frac \lambda 2+j) \big)}{\prod_{j=0}^{q-1}(\frac n2+j)} \frac{(-1)^q(\sinh(\frac{\eta}{2}))^{2q}}{q!}\notag\\
=&1+\sum_{q=1}^{\infty}\frac{\prod_{j=0}^{q-1}\big((\rho+j)^2+\frac{\lambda^2}{4}) \big)}{\prod_{j=1}^{q}(m+j)} \frac{(-1)^q(\sinh(\frac{\eta}{2}))^{2q}}{q!},m=\rho-\frac 12=\frac n2-1,\notag\\
=&1+\sum_{q=1}^{\infty}\frac{\prod_{j=0}^{q-1}\big((2\frac{\rho+j}{\lambda})^2+1) \big)}{4^q\prod_{j=1}^{q}(m+j)} \frac{(-1)^q(\lambda\sinh(\frac{\eta}{2}))^{2q}}{q!}\notag\\
=&1+\sum_{q=1}^\infty\frac{(-1)^q(\lambda \sinh(\frac \eta 2))^{2q}\Bigg[1+\Big(\frac{n-1}{\lambda} \Big)^2\Bigg]\times \cdots \times \Bigg[1+\Big(\frac{n+2q-3}{\lambda} \Big)^2 \Bigg]}{4^q q!(m+1)(m+2)\cdots(m+q)}\label{EXP_FONC_SPEC}\\
=&1+\sum_{q=1}^\infty\frac{(-1)^q(\lambda \sinh(\frac \eta 2))^{2q}\Bigg[1+\Big(\frac{2m+1}{\lambda} \Big)^2\Bigg]\times \cdots \times \Bigg[1+\Big(\frac{2m+2q-1}{\lambda} \Big)^2 \Bigg]}{4^q q!(m+1)(m+2)\cdots(m+q)}.\notag
\end{align}
}

We emphasize that the representation \eqref{EXP_FONC_SPEC} will be actually used to investigate the behavior of the quantity $\Big(\hat f_{\frac1{\sqrt N}\otimes Z}(\lambda)\Big)^N $ in \eqref{SPLIT_TAILS} for the \textcolor{black}{term $\mathcal T_N^{11}(\eta)$}. From \eqref{TF_H}, \textcolor{black}{\eqref{SCALED_DENS_PROP}}, and since on the considered integration interval $\lambda\in [N^{\frac 14},c_0N^{\frac 12}] $, this actually means that we will consider \eqref{EXP_FONC_SPEC} for a \textit{small} spatial argument. Namely,
\begin{align}\label{EXP_TF_DENS}
\hat f_{\frac1{\sqrt N}\otimes Z}(\lambda)=\Omega_{n-1}\int_0^\infty \varphi_\lambda(\textcolor{black}{\tanh(\frac{\eta}{2\sqrt N})}) f(\tanh(\frac \eta 2) )\sinh(\eta)^{n-1} d\eta =\int_0^\infty \varphi_\lambda\big(\textcolor{black}{\tanh(\frac{\eta}{2\sqrt N})}\big)   \mu_{Z,R}(d\eta).
\end{align}
Let us turn to the estimation of the r.h.s. in \eqref{EXP_FONC_SPEC}. \textcolor{black}{We will actually establish that}, \textcolor{black}{for sufficiently large $q$}:
\begin{equation}
\frac{\Bigg[1+\Big(\frac{\textcolor{black}{2m+1}}{\lambda} \Big)^2\Bigg]\times \cdots \times \Bigg[1+\Big(\frac{\textcolor{black}{2m+2q-1}}{\lambda} \Big)^2 \Bigg]}{4^q q!(m+1)(m+2)\cdots(m+q)}=\frac{m!\Bigg[1+\Big(\frac{\textcolor{black}{2m+1}}{\lambda} \Big)^2\Bigg]\times \cdots \times \Bigg[1+\Big(\frac{\textcolor{black}{2m+2q-1}}{\lambda} \Big)^2 \Bigg]}{4^q q!(m+q)!}\le C(m)4^{-q}.\label{U_BOUND_QUOTIENT}
\end{equation}
To establish the above inequality we first observe that
\begin{align}
\Bigg[1+\Big(\frac{\textcolor{black}{2m+1}}{\lambda} \Big)^2\Bigg]\times \cdots \times \Bigg[1+\Big(\frac{\textcolor{black}{2m+2q-1}}{\lambda} \Big)^2 \Bigg]\le& \exp\Bigg(q \ln\Big(1+\Big(\frac{\textcolor{black}{2m+2q-1}}{\lambda} \Big)^2 \Big)\Bigg)\notag\\
\le &\exp\Bigg( q\ln\Big(q^2(\frac 1{q^2}+\frac{C(m)}{\lambda^2})\Big)\Bigg)\le \exp(q\ln( [\frac{q}{e}]^2)) \notag \\
=&\exp(-2q)\exp (2q\ln q )),
\label{NUM}
\end{align}
where the last inequality holds for \textcolor{black}{$q\ge 3$} and \textcolor{black}{$N  $ large enough (recall $\lambda\in [N^{1/4},c_0N^{1/2}] $)}. On the other hand, from Stirling's formula we get
:
\begin{align}
q!(m+q)!&=2\pi \textcolor{black}{(q(m+q))^{\frac 12}}e^{q\ln q}e^{(m+q)\ln(m+q)}e^{-(m+2q)}e^{\frac1{12q}(\theta_{1q}+\theta_{2q})}, \ \theta_{iq}\in (0,1),\ i\in \{1,2\},\notag\\
&>2\pi e^{-m}e^{2q\ln q }e^{-2q},\label{DEN}
\end{align}
for sufficiently large $q$. From \eqref{DEN} and \eqref{NUM}, we derive \eqref{U_BOUND_QUOTIENT}.

On the other hand, on the considered range for $\lambda $, i.e. $\lambda \in [N^{\frac 14},c_0N^{\frac 12} ]$ we have:
$$\lambda \sinh(\frac \eta{\textcolor{black}{2}\sqrt N}) \le \lambda \frac \eta{\sqrt N}\le c_0\bar R<1,$$
choosing $c_0$ small enough and exploiting that $\mu_Z $ in \eqref{EXP_TF_DENS} is compactly supported on $[0,\bar R] $. We then derive that on the considered range
:
\begin{align}
\varphi_\lambda\big(\textcolor{black}{\tanh(\frac \eta{2\sqrt N})}\big)\le 1- \frac{\textcolor{black}{\lambda^2}\Big(\frac \eta{\sqrt N}\Big)^2}{\textcolor{black}{8}(m+1)}.\label{BD_PHI_LAMBDA_FIRST_AFTER_BULK}
\end{align}
Hence, from \eqref{EXP_TF_DENS} and  \eqref{BD_PHI_LAMBDA_FIRST_AFTER_BULK} we obtain:
\begin{align*}
|\hat f_{\frac 1{\sqrt N}\otimes Z}(\lambda)|\le (1-c\frac{\lambda^2}N \int_{\mathbb B^n} \eta^2\mu_{Z,R}(d\eta)),\ |\hat f_{\frac 1{\sqrt N}\otimes Z}(\lambda)|^N\le \exp(-c\lambda^2\int_{\mathbb B^n}\eta^2\mu_{Z,R}(d\eta))=\exp(-c\lambda^2t),
\end{align*}
\textcolor{black}{up to a modification of $c$ for this very last inequality}.
We then eventually get from \eqref{SPLIT_TAILS}, \textcolor{black}{analogously to \eqref{CTR_TAILS_HK}}, that:
\begin{align}
\textcolor{black}{\mathcal T_N^{11}(\eta)}&\le C\exp(-c\frac{(-\rho^2+\textcolor{black}{N^{\frac 12}})t}{ 2})\int_{N^{\frac 14}}^{c_0N^{\frac 12}}  \exp(-c\frac{(\rho^2+\lambda^2)t}{ 2})|\varphi_\lambda (\tanh(\frac \eta 2))||\c(\lambda)|^{-2}d\lambda\notag\\
&\le \frac{\textcolor{black}{C}}{t^2N }\textcolor{black}{\Big(\frac 1{t^{\frac 12}}\wedge \frac{1}{t^{\frac n2}} \Big)}. \label{CTR_T_N_11}
\end{align}

It now remains to consider the case $\lambda >c_0N^{\frac 12} $ to handle $\textcolor{black}{\mathcal T_N^{12}(\eta)} $ in \eqref{SPLIT_TAILS}. To this end we exploit, \textcolor{black}{see e.g. Lemma 5.1 formula (5.3) in \cite{petr:92}}, that
\begin{equation}\label{ASYM_VARPHI_LAMBDA}
 \exists c\in (0,1), \forall \lambda,r\in \R_+^2,\ \lambda r>1, |\varphi_\lambda (\textcolor{black}{\tanh(r)})|<1-c.
\end{equation}
Write then,
\begin{align*}
\textcolor{black}{\mathcal T_N^{12}(\eta)}&=\Big|\int_{c_0N^{\frac 12}}^{\infty}\Big(\prod_{j=1}^{N}\hat f_{\frac1{\sqrt N}\otimes Z}(\lambda)\Big) \varphi_\lambda (\tanh(\frac \eta 2))|\c(\lambda)|^{-2}d\lambda\Big|\le 
\int_{c_0N^{\frac 12}}^{\infty}|\hat f_{\frac 1{\sqrt N}\otimes Z}(\lambda)|^N |\varphi_\lambda (\tanh(\frac \eta 2))| |\c(\lambda)|^{-2}d\lambda\\
&\le \frac 1 b \int_{bc_0N^{\frac 12}}^{\infty}|\hat f_{\frac 1{\sqrt N}\otimes Z}(\frac \lambda b)|^N |\varphi_{\frac \lambda b} (\tanh(\frac \eta 2))| |\c(\frac \lambda b)|^{-2}d\lambda,
\end{align*}
for some parameter $b>0$ to be specified.
Let us now write \textcolor{black}{from \eqref{EXP_TF_DENS}}:
\begin{align}
\hat f_{\frac 1{\sqrt N}\otimes Z}(\frac \lambda b)=\int_0^{2b\frac{\sqrt N}{\lambda}}\varphi_{\frac \lambda b}(\tanh(\frac{\eta}{2\sqrt N}))\mu_{Z,R}(d\eta)+\int_{2b\frac{\sqrt N}{\lambda}}^{\bar R}\varphi_{\frac \lambda b} (\tanh(\frac{\eta}{2\sqrt N}))\mu_{Z,R}(d\eta)=:I+II.\label{DECOUP_LAST_STEP_TF}
\end{align}
 For the  term $II$ in \eqref{DECOUP_LAST_STEP_TF}, observe that for $\eta> 2b \frac{\sqrt{N}}\lambda$ then $ \frac \lambda b \frac{\eta}{2\sqrt N }>1$ and by \eqref{ASYM_VARPHI_LAMBDA},
 $$II\le 1-c .$$
 For the term $I$ observe that $2b\frac{\sqrt N}{\lambda}\le 2b \frac{\sqrt N}{c_0\sqrt N}=\frac {2b}{ c_0} $. Since $\mu_{Z,R} $ has a density there is in particular no atom at 0 and therefore, since 
 $\textcolor{black}{|\varphi_{\frac \lambda b} (\tanh(\frac{\eta}{2\sqrt N}))| \le C}$ 
  on the considered integration range, we then derive $\int_0^{\frac{2b}{c_0}}|\varphi_{\frac \lambda b} (\tanh(\frac{\eta}{2\sqrt N}))| \mu_{Z,R}(d\eta)<c/2$ for $b$ small enough. Hence, plugging those controls into \eqref{DECOUP_LAST_STEP_TF}, we get:
 $$|\hat f_{\frac 1{\sqrt N}\otimes Z}(\frac \lambda b)|\le 1-c/2:=\delta(b)<1. $$
 Therefore,
 \begin{align*}
\textcolor{black}{\mathcal T_N^{12}(\eta)}\le& \frac C b \delta^{N-2}\int_{bc_0 N^{\frac 12}}^{\infty}|\hat f_{\frac 1{\sqrt N}\otimes Z}(\frac \lambda b)|^2|\varphi_{\frac \lambda b} (\tanh(\frac \eta 2))| |\c(\frac \lambda b)|^{-2}d\lambda\\
=&C\delta^{N-2}\int_{c_0 N^{\frac 12}}^{\infty}|\hat f_{\frac 1{\sqrt N}\otimes Z}( \lambda )|^2|\varphi_{ \lambda } (\tanh(\frac \eta 2))| |\c( \lambda )|^{-2}d\lambda.
 \end{align*}
 We can now use the Plancherel equality, see e.g. Theorem 6.14 in \cite{liu:peng:04} with $\vartheta=2-n$, 
 to derive:
 \begin{align*}
 \textcolor{black}{\mathcal T_N^{12}(\eta)}&\le C\delta^{N-2}\int_0^{\bar R} f_{\frac 1{\sqrt N}\otimes Z}^2(\tanh(\frac \eta 2)) \sinh^{n-1} (\eta)d\eta\\
 &\le C\delta^{N-2} \int_0^{\frac{\bar R}{N^{\frac 12}}}
N f_Z^2(\tanh(\frac{N^{\frac 12}\eta}{2}))\Big(\frac{\sinh( N^{\frac 12}\eta )}{\sinh ( \eta)}\Big)^{2(n-1)}\sinh^{n-1} (\eta)d\eta\\
&\le C\delta^{N-2}N^{\frac 12}\int_0^{\bar R} f_Z^2(\tanh(\frac{\eta}2))\sinh(\eta)^{2(n-1)} (\sinh(\frac{\eta}{N^{\frac 12}}))^{-(n-1)}d\eta
 \end{align*}  
 using as well \eqref{SCALED_DENS_PROP} with $\varepsilon=N^{-\frac 12} $ for the second inequality. Hence,
 \begin{align*}
\textcolor{black}{\mathcal T_N^{12}(\eta)}&\le C\textcolor{black}{N^{\frac n2}}\delta^{N-2}\int_0^{\bar R} f_Z^2(\tanh(\frac{\eta}2))\sinh(\eta)^{2(n-1)} \eta^{-(n-1)}d\eta\le C\textcolor{black}{N^{\frac n2}}\delta^{N-2}=C\textcolor{black}{N^{\frac n2}}\exp( (N-2) \ln(\delta))\\
&\le \frac{C}N \exp(-\frac N 2|\ln(\delta)|),
 \end{align*}
 recalling that $\delta<1 $ for the last inequality. We have thus established from the above control and \eqref{CTR_T_N_11} that,
 \begin{align*}
\textcolor{black}{\mathcal  T_N^1(\eta)}\le \frac CN\Big( \frac 1{\textcolor{black}{t^{2}}}\textcolor{black}{\Big(\frac 1{t^{\frac 12}}\wedge \frac{1}{t^{\frac n2}} \Big)}+1\Big).
 \end{align*}
 From the above control and \eqref{CTR_TAILS_HK}, \eqref{DECOUP_TAILS_GEN} we thus derive:
 $$\textcolor{black}{\mathcal  T_N(\eta)}\le \frac CN\Big( \frac 1{\textcolor{black}{t^{2}}}\textcolor{black}{\Big(\frac 1{t^{\frac 12}}\wedge \frac{1}{t^{\frac n2}} \Big)}+1\Big),
 $$
\textcolor{black}{ which together with \eqref{CTR_BULK} and \eqref{THE_DECOUP_LLT} completes the proof of the local limit Theorem \ref{LLT}.}
\end{trivlist}
 


 \bibliographystyle{alpha}

\bibliography{bibli}

\end{document}